\documentclass[11pt, reqno]{amsart}
\usepackage{geometry}
\usepackage{amsmath, amssymb,color, amsthm, soul, enumitem, cancel,url}
\usepackage[utf8]{inputenc}
\numberwithin{equation}{section}
\usepackage{hyperref, xcolor}
\definecolor{goetheblau}{cmyk}{1.00 0.20 00 0.40}
\definecolor{hellgrau}{cmyk}{0.04 0.04 0.05 0.02}
\definecolor{sandgrau}{cmyk}{0.12 0.09 0.13 0}
\definecolor{dunkelgrau}{cmyk}{0.25 0.25 0.30 0.75}
\definecolor{purple}{cmyk}{0.08 1.00 0.30 0.36}
\definecolor{emorot}{cmyk}{0.04 1.00 0.80 0.07}
\definecolor{senfgelb}{cmyk}{0.01 0.25 1.00 0.05}
\definecolor{gruen}{cmyk}{0.62 0.40 0.87 0.09}
\definecolor{magenta}{cmyk}{0.08 0.86 0.12 0.12}
\definecolor{orange}{cmyk}{0 0.70 1.00 0.04}
\definecolor{sonnengelb}{cmyk}{0 0.12 0.95 0}
\definecolor{hellesgruen}{cmyk}{0.40 0.17 0.81 0.07}
\definecolor{lichtblau}{cmyk}{0.80 00 0.06 0.04}
\definecolor{teal}{rgb}{0.0, 0.51, 0.5}

\hypersetup{%
  colorlinks = true,
  linkcolor  = emorot,
  citecolor = hellesgruen
}

\usepackage{mathrsfs}
\usepackage{commath}
\usepackage{esdiff}
\usepackage{psfrag}
\usepackage{soul}

\newtheorem{theorem}{Theorem}[section]

\newtheorem{lemma}[theorem]{Lemma}
\newtheorem{proposition}[theorem]{Proposition}

\theoremstyle{remark}
\newtheorem{remark}[theorem]{Remark}

\theoremstyle{definition}
\newtheorem{definition}[theorem]{Definition}
\newcommand{\N}{\mathbb{N}}

\newcommand{\R}{\mathbb{R}}

\newcommand{\eps}{\varepsilon}

\DeclareMathOperator{\E}{\mathbf{E}}

\renewcommand{\P}{\mathbf{P}}

\usepackage{pgfplots}
\usepackage{tikz}
\usetikzlibrary{arrows.meta,bending}
\usetikzlibrary{patterns}
\usepackage{graphicx}
\newlength\figureheight
\newlength\figurewidth

\usepackage{framed}
\usepackage{todonotes}
\usepackage{graphicx}
\usepackage{subcaption}
{\endMakeFramed}

\title[Click times of Muller's ratchet]{
 The tournament ratchet's clicktime process,\\  and metastability in a Moran model}

\author[J.L. Igelbrink]{Jan Lukas Igelbrink}
\address{Jan Lukas Igelbrink}
\email{igelbrin@math.uni-frankfurt.de}

\author[C. Smadi]{Charline Smadi}
\address{Charline Smadi, Univ. Grenoble Alpes, INRAE, LESSEM, 38000 Grenoble, France
 and Univ. Grenoble Alpes, CNRS, Institut Fourier, 38000 Grenoble, France}
\email{charline.smadi@inrae.fr}

\author[A. Wakolbinger]{Anton Wakolbinger}
\address{Anton Wakolbinger, Goethe-Universit\"at, Institut f\"ur Mathematik, 60629 Frankfurt am Main, Germany}
\email{wakolbinger@math.uni-frankfurt.de}

\mathcode`l="8000
\begingroup
\makeatletter
\lccode`\~=`\l
\DeclareMathSymbol{\lsb@l}{\mathalpha}{letters}{`l}
\lowercase{\gdef~{\ifnum\the\mathgroup=\m@ne \ell \else \lsb@l \fi}}%
\endgroup

\subjclass[2020]{Primary 92D15; secondary 60K35,  60J80, 60J27, 60F17}
\keywords{Muller's ratchet, click rate, tournament selection, Moran model, metastability}

\begin{document}
\begin{abstract}
 Muller's ratchet, in its prototype version, models a haploid, asexual population whose size~$N$ is constant over the generations. Slightly deleterious mutations are acquired along the lineages at a constant rate, and individuals carrying less mutations have a selective advantage. In the classical variant, an individual's selective advantage is proportional to the difference between the population average and the individual's mutation load, whereas in the ratchet with {\em tournament selection}  only the signs of the differences of the individual mutation loads matter. In a parameter regime which leads to slow clicking (i.e. to a loss of the currently  fittest class at a rate $\ll 1/N$) we  prove that the rescaled process of click times  of the tournament ratchet converges  as $N\to \infty$  to a Poisson process. Central ingredients in the proof are a thorough analysis of the  metastable behaviour of a two-type Moran model with selection and deleterious mutation (which describes the size of the fittest class up to its extinction time) and a lower estimate on the size of the new fittest class at a clicktime.
\end{abstract}

\normalcolor
\allowdisplaybreaks
\maketitle
\noindent
\tableofcontents
\section{Introduction}  
\label{sec_intro_results}
\noindent{\em Muller’s ratchet} is a prototype model in population genetics; among the pioneering papers are \cite{H78,Stephan1993TheAO,Gordo2000TheDO}. Originally, this model was conceived to explain the ubiquity of sexual reproduction among eukaryotes despite its many costs \cite{M64,felsenstein1974evolutionary}. In its bare-bones version, Muller's ratchet models a haploid asexual
population whose size $N$ is constant over generations.  Slightly deleterious mutations are acquired along the lineages of descent at a constant rate, and an individual's mutational load (i.e.~the number of mutations that was accumulated along the individual's ancestral lineage) compared to the mutational loads of its contemporaries is decisive for the individual's success in the selective part of the reproduction dynamics.

 With the type of an individual being its current mutational load,  the population dynamics leads to a type frequency profile that is ``driven upwards'' by the effect of mutation and ``kept tight'' by the effect of selection. Due to  mutation and randomness in the individual reproduction, the currently fittest class, i.e.~the subpopulation of individuals carrying the currently smallest mutational load,
 will eventually get extinct; this is a {\em click of the ratchet}. 
\subsubsection*{\bf Tournament versus fitness proportionate  selection}
 For quantifying the effect of an individual mutational load within a population there are two prototypic ways: in the so-called {\em fitness proportionate selection}, the {\em value} of the individual mutational load is compared to the {\em mean value} of the mutational loads in the population, and the individual's selective advantage is proportional to the resulting difference. In the so-called {\em tournament selection}, the selective advantage of an individual is determined by the  {\em rank} of its mutational load within the contemporary population. We will work with the (continuous-time) Moran model for Muller's ratchet with tournament selection, and refer to it as the {\em tournament ratchet} for short. 
Here we  give a brief verbal description of the individual-based, Poisson-process driven dynamics of the tournament ratchet  as specified by the graphical representation  in~\cite{GSW23}, which is in line with the definition of the type frequency process given at the beginning of Section~\ref{sec_model}.

\smallskip
 {\em - For each ordered pair of individuals, neutral reproduction events come at rate $\tfrac 1{2N}$, resulting in a binary reproduction of the first and the death of the second individual.}

 \smallskip
{\em  - For each pair of individuals carrying different mutational loads, selective reproduction events come at rate $s_N/N$, resulting in a binary reproduction of the fitter and the death of the less fit individual.} 

\smallskip
 {\em - Each  individual's mutation load is increased by 1 at rate $m_N^{}$}. 

\smallskip
\noindent 
An appealing feature of the tournament ratchet is that for each $k\in \mathbb N$ the dynamics of the $k$~fittest classes is autonomous up to the time of extinction of the fittest class. In particular, the size of the currently fittest class is described by a two-type Moran model with selection and deleterious mutation.
This ``hierarchical autonomy'' of the fitness classes does not hold true for  Muller's ratchet with fitness proportionate selection, and to the best of our knowledge the asymptotic analysis of the click rate of the latter so far has resisted a complete  and rigorous solution  
despite several attacks, among them \cite{EPW09, neher2012fluctuations, MPV20}. 
Links between the two ratchet models are established in \cite{IGSW}, showing that under an appropriate transformation of the mutation-selection ratio $\rho_N:=m_N/s_N$ in the {\em near-critical parameter regime} $\rho_N \uparrow 1$  the dynamics of the size of the fittest class of the tournament ratchet becomes what is called in~\cite{EPW09} the Poisson profile approximation of the size of the fittest class of Muller's ratchet with fitness proportionate selection.
\subsubsection*{\bf Brief summary of results} 
We focus on a parameter regime of moderate mutation and selection,  with\footnote{As usual, for two positive sequences $(a_N)$, $(b_N)$, the notation $a_N \gg b_N$ means  that $a_N/b_N \to \infty$, and  $a_N\sim b_N$ means that the two sequences are {\em asymptotically equivalent}, i.e. $a_N/b_N \to 1$.} $1\gg s_N > m_N \gg 1/N$, $m_N/s_N = \rho_N$ being convergent as $N\to \infty$, and satisfying the condition
$\mathfrak a_N/\mathfrak c_N\to \infty$, where $\mathfrak a_N := N(1-\rho_N)$ is the center of attraction of the size of the fittest class and $\mathfrak c_N := (s_N-m_N)^{-1}$ is the critical size below which the fittest class (then similar to a slightly supercritical branching processs) experiences quick extinction.   Specifically, it turns out that, after a click followed by a relaxation time of duration  $O(\mathfrak c_N \log \mathfrak a_N)$, the size of the fittest class performs (asymptotically Ornstein-Uhlenbeck) fluctuations on the $\mathfrak c_N$-timescale, and goes extinct after an asymptotically exponentially distributed time on a much larger timescale, then with the next-fittest class taking over. In this sense, under the condition $\mathfrak c_N \ll \mathfrak a_N$, the model exhibits a metastable behaviour. (For more background on the concept of metastability, see e.g. the monograph \cite[Chapter 8]{bovier2016metastability}.)

Thus our main result concerning the tournament ratchet (Theorem~\ref{th:clicktimepois}) is the convergence of the rescaled clicktime process to a standard Poisson process, with   the expected time between clicks  being given up to asymptotic equivalence as $N\to  \infty$ by an explicitly computed   quantity $e_N$. This quantity  appeared in~\cite{IGSW} via a Green function analysis as the asymptotic equivalent of the expected extinction time of the fittest class when started from large sizes, see eq.~\eqref{asympeN} below. 

Theorem~\ref{th:clicktimepois} constitutes also a substantial improvement compared to~\cite[Theorem 2.2]{GSW23}, where convergence of the rescaled clicktime process to a standard Poisson process was proved in the special case of the mutation-selection ratio  $m_N/s_N$ not depending on $N$, and where the click rate was determined only up to logarithmic equivalence. The methods applied in the present paper are fundamentally different from those of~\cite{GSW23}: while that paper pursued a backward-in-time strategy (relying on the  duality of the tournament ratchet with a hierarchy of decorated ancestral selection graphs), the present paper takes a forward-in-time approach. As indicated above, a key role  is played by the metastable behaviour of a two-type Moran model with one-way mutation; our corresponding results are subsumed in Theorem~\ref{thmY0}. 

While the present work focusses on the evolution of the fittest class of the tournament ratchet, it is also of interest to study the type frequency profile of the entire population. This was achieved in  \cite{GSW23} for $\rho < 1$ not depending on $N$.  
In a forthcoming  paper we will extend this to the more delicate case $\rho \uparrow 1$.
Like in \cite{GSW23}, a central tool for analysing the empirical type frequency profile will be a hierarchy of decorated ancestral selection graphs, which extends the decorated ASG of the Moran model with selection and (one-way) mutation (for the latter see e.g.~\cite{baake2018lines} and references therein).

\section{Model and main results}\label{sec_model}
\subsection{The tournament ratchet and its clicktime process}\label{mainconc}
We start by defining the state space and the transition rates of the type frequency process which result from the individual-based dynamics described in the Introduction.
\begin{definition}\label{tfpdyn}
    For given parameters $N\in \N$, and $m_N$, $s_N > 0$, let $\mathfrak N^{(N)}= (\mathfrak N^{(N)}_\kappa)=(\mathfrak N^{(N)}_\kappa(t))$, $\kappa\in \N_0$, $t \ge 0$, be a Markovian jump process with state space $E_N:=\{(n_0, n_1, \ldots)\, | \,$ \mbox{$n_0, n_1, \ldots\in \N_0$}, $n_0+n_1+\cdots = N\}$ and with the following transition rates: 

\smallskip
- Neutral reproduction: for $\kappa \neq \kappa'$,    \\
\phantom{AAA} $(\mathfrak N_\kappa, \mathfrak N_{\kappa'})$ jumps to $(\mathfrak N_\kappa+1, \mathfrak N_{\kappa'}-1)$ at rate $\frac 1{2N} \mathfrak N_\kappa \mathfrak N_{\kappa'}$ 

\smallskip
- Selective reproduction: for $\kappa < \kappa'$, \\
\phantom{AAA}  $(\mathfrak N_\kappa, \mathfrak N_{\kappa'})$ jumps to $(\mathfrak N_\kappa+1, \mathfrak N_{\kappa'}-1)$ at rate $\frac { s_N^{}}N \mathfrak N_\kappa \mathfrak N_{\kappa'}$,

\smallskip
- Mutation: for $\kappa$, \\
\phantom{AAA} $(\mathfrak N_\kappa, \mathfrak N_{\kappa+1})$ jumps to $(\mathfrak N_\kappa-1, \mathfrak N_{\kappa+1}+1)$ at rate $m_N^{} \mathfrak N_\kappa   $.

\end{definition}
  
As in this definition with $\mathfrak N_\kappa= \mathfrak N_\kappa^{(N)}$, we will often suppress the dependence on $N$ to ease  the notation. As motivated in the Introduction, we will refer to a process following the dynamics specified in Definition~\ref{tfpdyn} as the {\em (type frequency process of the) tournament ratchet with population size $N$, selection parameter $s_N$ and mutation rate $m_N$}. Even though the space~$E_N$ is infinite, it is easy to see (e.g. from the graphical construction provided in \cite{GSW23}) that the process $\mathfrak N^{(N)}$ is well-defined. We will address $\kappa$ as the {\em type} or {\em mutation load} carried by an individual; consequently, $\mathfrak N^{(N)}_\kappa(t)$ is the size of the subpopulation (or {\em class}) of the individuals of type $\kappa$ that live at time $t$.
According to Definition~\ref{tfpdyn}, individuals carrying a smaller mutation load are selectively favoured. 
The {\em type of the fittest class at time } $t\ge 0$ is denoted by 
\begin{equation}\label{eq:besttypeKstart}
    K_N^\star(t):= \min\left\{  \kappa: \mathfrak N^{(N)}_{\kappa}(t)>0\right\}.
\end{equation}
The times $\mathcal T_N^{(i)}$, $i=1,2,\ldots$, when the currently fittest class is lost forever,
\begin{equation}\label{def_Ti}
    \mathcal T_N^{(i)}:= \inf \left\{ t: K_N^\star(t)=i \right\},
\end{equation}
are called the  {\em clicktimes} of the ratchet.
For $\ell\in \N_0$ and $t\ge 0$  we put
\begin{equation}\label{defYk}
\mathfrak N_\ell^\star(t) := \mathfrak N_\ell^{(N)\star}(t) :=  \mathfrak N_{K_N^\star(t)+\ell}^{(N)}(t).
\end{equation}
In words, $\mathfrak N_\ell^\star(t)$ is the size of the subpopulation of individuals that live at time $t$ and carry $\ell$ mutations more than those in the fittest class at time $t$. 
\begin{remark}\label{autonomy}As observed in~\cite{GSW23}, 
for each $\ell=0,1,\ldots$ the dynamics of $(\mathfrak N_0^\star, \ldots, \mathfrak N_\ell^\star)$ is {\em autonomous}  up to the next clicktime (at which the ``old'' fittest class goes extinct and $\mathfrak N_0^\star$ re-starts as the size of the ``new'' fittest class).  In particular, $\mathfrak N_0^\star$, the {\em size of the currently fittest class}, has between click times $\mathcal T^{(i)}$ and  $\mathcal T^{(i+1)}$ (and with $\mathcal T^{(i)}$ shifted to the time origin)   the same distribution as a birth-death process~$Y_0$ up to the time of its extinction, with $Y_0(0) = \mathfrak N_1^\star(\mathcal T^{(i)}-)$ for $i\ge 1$, and $Y_0(0)\gg \mathfrak c_N$ for $i=0$ under assumption~\eqref{incondnew}. The  upward and downward jump rates of $Y_0$ from state $n\le N$  are  given by 
\begin{equation}\label{bestupproof}
\lambda_n:=\lambda^{(N)}_n:= n\left(\frac 12\left(1-\frac nN\right) + s \left(1-\frac nN\right)\right),
\end{equation}
\begin{equation}\label{bestdownproof}
\mu_n: = \mu_n^{(N)} := n\left(\frac 12\left(1-\frac nN\right)+m\right).
\end{equation} 
This is the dynamics of the number of type-0 individuals in a two-type Moran model with one-way mutation (from type 0 to type 1) at rate $m$ per individual, and with individual selective advantage $s/N$ of type 0 against type 1; we will come back to this in the next subsection. 
\end{remark}
With $N,m_N, s_N$ as in Definition~\ref{tfpdyn}, we will write
$$\rho_N:= m_N/s_N$$ for the {\em mutation-selection ratio}. 
The following will be assumed as $N\to \infty$: 
\begin{equation}\label{smallmut}
  \tfrac 1N \ll m_N < s_N\ll 1 \quad \mbox{and }  \quad \rho_N \to \rho^\ast \in (0,1].
\end{equation}
Loosely spoken, assumption~\eqref{smallmut} says that mutation is ``moderate'', acting on a timescale that is intermediate between the ecological and the evolutionary timescale, and that selection acts on the same timescale as mutation. 
 We now introduce (and interpret) two quantities $\mathfrak a$ and $ \mathfrak c$ whose interplay will turn out to be relevant for the long-term behaviour of $Y_0$. With
\begin{equation}\label{def_a}
  \mathfrak a =  \mathfrak a_N :=  N (1-\rho_N)
\end{equation}
the drift $\lambda_n-\mu_n$ takes the form 
\begin{equation}\label{drift}
    \lambda_n - \mu_n = {n}(s_N-m_N)\Big(1 - \frac{n}{\mathfrak a}\Big),
\end{equation}
hence $\mathfrak a$
is the {\em center of attraction} of $Y_0$. 
The quantity $s_N-m_N$ is the supercriticality in the branching process that approximates $Y_0$ as long as $Y_0 \ll \mathfrak a$;  consequently the quantity
\begin{equation}\label{defh}\mathfrak c = \mathfrak c_N^{} = \frac 1{s_N^{}-m_N^{}} =   \frac {\rho_N^{}}{m_N^{}(1-\rho_N^{})}
\end{equation}
is the {\em critical size}  below which $Y_0$ will quickly go to extinction. In this work we will focus on the condition  
\begin{equation}\label{clla}
 \mathfrak c_N \ll \mathfrak a_N \, \mbox{ as }N \to \infty.   
\end{equation} 
Under condition~\eqref{smallmut}, the requirement~\eqref{clla} is  equivalent to 
\begin{equation}\label{expreg}
     \mathfrak u_N:=\frac{\mathfrak a_N}{\mathfrak c_N}\rho_N = N m_N (1-\rho_N)^2 \rightarrow \infty.
\end{equation}Following~\cite{IGSW} we will refer to parameter constellations satisfying condition~\eqref{expreg} as the {\em exponential regime}.
For
$$T_0 := T_0^{(N)}:= \mbox{the time of extinction of } Y_0^{(N)}$$ 
it was proved in~\cite[Theorem 3.4 part b)]{IGSW} that
\begin{equation} \label{asympeN}
    \E\left[T_0\, \big| \,Y_0^{(N)}(0) = j_N\right]
        \sim e_N \qquad \mbox{if } j_N \gg \mathfrak c_N,
\end{equation}
where
\begin{equation}\label{eN}
e_N^{}:= \mathfrak c_N^{} \sqrt{\frac \pi{\mathfrak u_N^{}}}\exp\left(2\mathfrak u_N^{}\eta(m_N^{},\rho_N^{})\right),
\end{equation}
and
\begin{equation} \label{def_eta}
\eta(m,\rho):=  -\frac{1}{2m} \left[\frac{1}{1-\rho}\log \left( \frac{1+2m}{1+2m/\rho}  \right)+ \sum_{\ell=1}^{\infty} \left(1- \frac{1}{(1+2m)^\ell}\right) \frac{(1-\rho)^{\ell-1}}{\ell(\ell+1)}\right]. \end{equation}
\begin{theorem}\label{th:clicktimepois}Assume the conditions~\eqref{smallmut}  and~\eqref {expreg}.
Then, with $e_N$ as in~\eqref{eN},
and with the initial condition
\begin{equation}\label{incondnew}
    \mathfrak N_0^{(N)}(0) \gg  \mathfrak c_N \quad \mbox{ as } N\to \infty,
\end{equation}
the sequence of $\N_0$-valued processes $(K_N^\star(e_N^{}t))_{t\ge 0}$ converges in distribution as $N\to \infty$ to a rate 1 Poisson counting process.
In particular, the sequence of time-rescaled clicktime processes
  \begin{equation*}
    \left(  \mathcal T_N^{(i)}/ e_N  \right)_{i\in \mathbb N}, \quad N = 1,2,\ldots,
  \end{equation*}
converges in distribution as $N\rightarrow \infty$ to a rate 1 Poisson point process on $\R_+$. 
\end{theorem}
Key ingredients in the proof of this theorem are parts a)-d) of Theorem~\ref{thmY0} (stated in the next subsection and proved in Section~\ref{pfthm1}) as well as Proposition~\ref{prop:sizenewbestclass} which shows that with high probability,  shortly before the fittest class gets extinct, there are still enough mutations affecting this class so that, when the old fittest class disappears, the size of the new fittest class is large enough to escape quick extinction caused by random fluctuations.  In the next subsection we will focus on the long-term behavior of the processes $Y^{(N)}_0$ when started in a sufficiently large state $j_N$.
\subsection{Metastability in a two-type Moran model with selection and deleterious mutations} Let $Y_0 = Y_0^{(N)}$ be an $\{0,1,\ldots, N\}$-valued continuous-time birth-death process with jump rates given by~\eqref{bestupproof}, \eqref{bestdownproof}. As already observed, $Y_0^{(N)}$ describes the number of individuals carrying the beneficial allele $0$ in a two-type Moran model with mutation rate $m=m_N$ from type $0$ to type $1$ and selection coefficient $s=s_N$.   
Our main result on this model (parts of which are instrumental also for the proof of Theorem~\ref{th:clicktimepois}) concerns the asymptotic normality of the quasi-equilibrium of $Y_0^{(N)}$, the asymptotic exponentiality of the extinction time of $Y_0^{(N)}$, and the convergence of the properly rescaled $Y_0^{(N)}$ to an Ornstein-Uhlenbeck process.
Recall (see e.g.~\cite{MeleardVillemonais2012}) that
the {\em quasi-equilibrium} of $Y_0=Y^{(N)}_0$ is the (uniquely determined) probability distribution $\alpha_N$ on $\{1,\ldots, N\}$ for which   
\begin{equation}\label{QED}
    \P_{\alpha_N^{}}(Y_0(t) = k\, | \, Y_0(t) \neq 0) = \alpha_N(k), \qquad 1\le k\le N, \, t > 0. 
\end{equation}
The quantity
\begin{equation}\label{defb}
\sigma:= \sigma_N:=   
  \sqrt {\rho_N\mathfrak a_N \mathfrak c_N}= \mathfrak c_N\sqrt{\mathfrak u_N}  \, = \frac{\rho_N\mathfrak a_N}{\sqrt{u_N}} = \frac 1{s_N}\sqrt{Nm_N}
\end{equation} 
 will emerge as a scale parameter of the fluctuations of $Y_0^{(N)}$ around its center of attraction $\mathfrak a_N$, cf. Remark~\ref{remBGandrates} for a quick explanation of the form of $\sigma$.
\newpage
\begin{theorem}\label{thmY0} Assume the conditions~\eqref{smallmut}  and~\eqref {expreg}. Then 
\begin{enumerate}[label=\alph*)]
    \item  With $T_\mathfrak a$ denoting the time at which $Y_0^{(N)}$ first hits the state $\lfloor \mathfrak a_N \rfloor$, all three of  the expectations  $$\E[T_\mathfrak a \mid Y_0^{(N)}(0) = N], \, \E[T_\mathfrak a \mid Y_0^{(N)}(0) = 1, \, T_\mathfrak a <T_0], \, \E[T_0 \mid Y_0^{(N)}(0) = \lfloor\mathfrak a\rfloor - 1, \, T_0 < T_\mathfrak a]$$ are \, $O(\mathfrak c_N\log \mathfrak a_N)$. \label{thmY0:a}
\item With $Y_0^{(N)}$ started in $\lfloor \mathfrak a_N\rfloor$, the sequence of processes
\begin{equation}\label{OUconvdef} 
\mathcal H_N := \left(\frac 1{\sigma_N^{}}\left(Y_0^{(N)}(t \mathfrak c_N)-\mathfrak a_N\right)\right)_{t\ge 0}
\end{equation}
  converges in distribution to a standard Ornstein-Uhlenbeck process started in the origin,
  i.e. to the process $\mathcal H$ satisfying the SDE
\begin{equation}\label{OUSDE}
d\mathcal H = -\mathcal H\, dt +d\mathcal W, \qquad \mathcal H(0) = 0,
\end{equation}
with $\mathcal W$ being a standard Wiener process.
\label{thmY0:b}
   \item The sequence of quasi-equilibria $\alpha_N$ of $Y_0^{(N)}$ is asymptotically normal as $N\to \infty$.  Specifically, the image of $\alpha_N$ under the mapping  $n\to (n-\mathfrak a_N)/\sigma_N$ converges weakly to $\mathcal N(0,1/2)$, the centered normal distribution with variance $1/2$.
\label{thmY0:c}
   \item With $Y_0^{(N)}$ started in $j_N \gg c_N$, the sequence $T_0^{(N)}/e_N$, $N=1,2,\ldots$, converges in distribution to a standard exponential random variable.
\label{thmY0:d}
\item For $\eps < 1/3$ let $L=L_\eps^{(N)}$ be the time at which $Y_0$ (having started in $j_N \gg \mathfrak c$) visits $\lfloor2\eps\mathfrak a_N\rfloor$ for the last time before going extinct. Then with high probability the number of mutations affecting the ``type 0'' population (whose size is $Y_0$) during period $[L, T_0]$ is $\gg \mathfrak c.$   \label{thmY0:e}
\end{enumerate}
\end{theorem}
\begin{remark}\label{Rem2_5}
a) In~\cite[Theorem 3.4 part b)]{IGSW} it is shown that the factor $v_N:= \mathfrak c_N\exp\left(2\mathfrak u_N\eta(m_N,\rho_N)\right)$ in~\eqref{eN} is  asymptotically equivalent to the expected number of returns of  $Y_0$ from~$\lfloor \mathfrak a\rfloor$  to~$\lfloor\mathfrak a\rfloor$.

    b) It follows from~\eqref{defb} that~\eqref{eN} can be written as
\begin{align}\label{eNnew}
    e_N^{} =\frac {1}{\rho_N^{} \mathfrak a_N^{}} \cdot \sqrt{\pi} \,\sigma_N^{} \cdot \mathfrak c_N^{}\exp\left(2u_N^{}\, \eta(m_N^{},\rho_N^{})\right).
\end{align}
The three factors in~\eqref{eNnew} can be interpreted  as the asymptotics of the expected holding time of $Y_0^{(N)}$ in states close to $\mathfrak a_N$ (see Remark~\ref{remBGandrates} b)), the expected duration of an excursion of $Y_0^{(N)}$ fom $\lfloor \mathfrak a_N\rfloor $ in its quasi-equilibrium (cf.~Theorem~\ref{thmY0}.\ref{thmY0:a}) and the expected number of such excursions before $Y_0^{(N)}$ escapes to $0$ (see Lemma~\ref{chalem} A) with $\mathfrak d_N = 1$).

c) Parts b) and c) of Theorem~\ref{thmY0} together imply that the sequence of quasi-equilibria of $\mathcal H_N$   converges weakly to the equilibrium distribution of the Ornstein-Uhlenbeck process $\mathcal H$. 

d) For proving part d), i.e.~the asymptotic exponentiality of the extinction times~$T_0^{N}$, there is also an alternative route which does not direcly invoke the quasi-equilibria $\alpha_N$ but instead uses \cite[Theorem 1]{grubel2005rarity}, see Remark~\ref{altern}.  In the present work we take a route via the analysis of quasi-equilibria, also because we think that this is interesting in its own right.
\end{remark}
\begin{remark} Related results for processes similar to $Y_0^{(N)}$ appear in the literature:

a) Consider, e.g., the case in which $m$ and~$s$ (other than in~\eqref{smallmut})  do not depend on $N$ (i.e. both mutation and selection act on the ``generation timescale''). Then, with $Y_0$ started in $\mathfrak a$,  the sequence of processes $ {\frac 1{\sqrt N}} ( Y_0-\mathfrak a)$ converges in distribution to an OU process as $N\to \infty$ (see~\cite{C17Markovproc} and references therein).

b)
    The ``drift''~\eqref{drift}
is that of a logistic branching process with carrying capacity $\mathfrak a$ and intrinsic growth rate $s-m$. The ``speed'' $\lambda_n+\mu_n$, however, is different from that of a logistic branching process, which prevents an application of results from \cite{ALambert} or~\cite{sagitov2013extinction}.  

c) The dynamics of $Y_0$ also bears resemblance to the so-called {\em logistic birth-and-death processes}, see \cite{foxall2021extinction} and references therein. However, in these processes the ``resampling'' terms $\tfrac 12(1-\tfrac nN) $ are missing in the rates $\lambda_n$ and $\mu_n$, which requires a control of their effects and also leads to subtle differences between the formulas for the expected extinction time in ~\eqref{eN} and in \cite[Theorem 5.3]{foxall2021extinction}.

d) For a class of processes containing the logistic birth-and-death processes and the logistic branching processes, a Gaussian approximation of the quasistationary distributions is obtatined in~\cite{chazottes2016sharp} by matching techniques reminiscent of  the WKB method from Physics, see also the references at the end of the Introduction of~\cite{chazottes2016sharp}.
In~\cite{chazottes2023large} the asymptotics of the spectrum of the rescaled generators of these processes is obtained as the superposition of the spectra of an Ornstein-Uhlenbeck operator and of a continuous-time binary branching process conditioned in non-extinction. It is conceivable that similarly fine results also hold for the family $Y_0^{(N)}$, even though  it does not quite fit into the scaling condition required in~\cite{chazottes2016sharp} and  in~\cite{chazottes2023large}. Also, the techniques applied in the present work may  shed additional light on the probabilistic background of the results of ~\cite{chazottes2023large}. 
\end{remark}
\normalcolor
\begin{remark}
    Within the regime~\eqref{smallmut},
\cite{IGSW} identifies two subregimes {with regard to the expected extinction times of $Y_0$}:    the {\em exponential regime}~\eqref{expreg}
and the {\em polynomial regime} in which 
\begin{equation*}
 \mathfrak u_N \to 0 \quad \mbox{ as } N\to 0. 
\end{equation*}
An extended version of the asymptotics~\eqref{asympeN} is given in \cite[eq.~(3.13)] {IGSW}. In that equation  the factor~$\frac 1{1-\rho}$ is lacking  in the journal publication. This (as well as details in the proof of \cite[Theorem~3.4 part a)]{IGSW} concerning the polynomial regime) was amended in the arxiv version~v3 of~\cite{IGSW}.
\end{remark}
   \section{Proof of Theorem~\ref{thmY0}}\label{pfthm1}
\subsection{Properties of the potential function}\label{propwell}
The analysis of expected hitting times of the process $Y_0 = Y_0^{(N)}$ relies on a study of functionals of the {\em oddsratio products} $r_\ell := r_\ell^{(N)}$  given by
\begin{equation} \label{def_ri} r_0=1, \quad r_l := \prod_{i=1}^l \frac{\mu_i}{\lambda_i}, \quad l \geq 1. \end{equation}
The {\em potential function} $U:= U^{(N)}$ is defined as
\begin{equation}\label{defU}
U(n):= \sum_{l=1}^n \log \frac{\mu_\ell}{\lambda_\ell}= \log r_n, \qquad n=0,1, \ldots, N,     \end{equation}
see Figure~\ref{fig:1} for an illustration.
\begin{figure}
    \centering
    \psfrag{a}{$\mathfrak{a}$}
    \psfrag{b}{\hspace*{-0.6cm}$\mathfrak{a}-\sigma$}
    \psfrag{d}{$\mathfrak{c}$}
    \psfrag{e}{\small$\sqrt{\frac{N}{m}}$}
    \begin{subfigure}[]{0.49 \textwidth}
\includegraphics[scale=0.9]{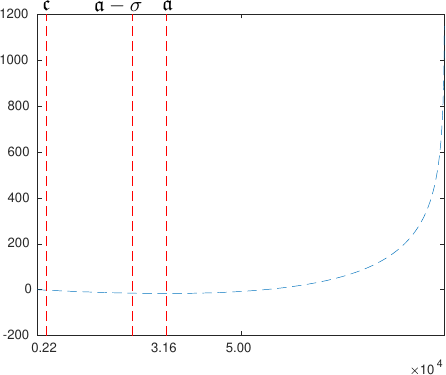}
    \end{subfigure}\hfill 
     \begin{subfigure}[]{0.49 \textwidth}
\includegraphics[scale=0.9]{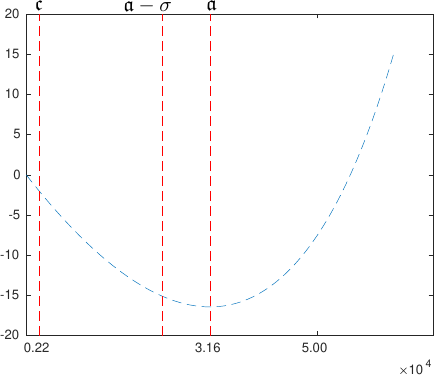}
    \end{subfigure}
    \caption{A plot of the potential function $U(n) = \sum_{\ell=1}^n \log\frac{\mu_\ell}{\lambda_\ell}$, $n\in [N]$,  for $N=10^5$, $m = N^{-0.6}$, $\rho = 1- N^{-0.1}=0.68$. 
    The left panel shows the full domain and range of~$U$,  the right panel restricts to $n\le 6.5\cdot 10^4$.  The quantities $\mathfrak a$, $\mathfrak c$ and~$\sigma$ are defined and explained in~\eqref{def_a}, \eqref{defh} and~\eqref{defb}.} 
\label{fig:1}
\end{figure} 
The function $R=R^{(N)}$ defined as  \begin{equation}\label{RKdef}
    R_k:= \sum_{i=0}^{k-1}r_i =\sum_{i=0}^{k-1}e^{U(i)}, \quad 0\le k \le N,
  \end{equation}
  is harmonic for the time-discrete birth-and-death chain associated with $Y_0$, cf.~\cite[Sec. 4.5]{IGSW}.
  \begin{proposition}\label{Uquad}
  With $\mathfrak a = \mathfrak a_N $, $\sigma=\sigma_N$ and $\mathfrak u_N$ defined in~\eqref{def_a}, \eqref{defb} and~\eqref{expreg} we have the following asymptotics as $N\to \infty$:      \begin{align}\label{Uas}
      U^{(N)}( \mathfrak a+K\sigma)-U^{(N)}(\mathfrak a)
= K^2(1+O(m_N + K/\sqrt{u_N})), \qquad K\in \mathbb R.
\end{align}
  \end{proposition}
  To prove~\eqref{Uas} we use a rescaling of $U^{(N)}$ that takes $\mathfrak a_N$ and $2\mathfrak u_N$ as units of population size and potential depth, respectively. 
In \cite{IGSW} it is proved that
the  potential function  can be represented as
\begin{equation}\label{link_HNH}
 U^{(N)}(n)= -2\mathfrak u_NH\left((m_N,\rho_N),\frac n{\mathfrak a_N}\right),
 \end{equation}
with $H((m,\rho),y)$ defined for $y \in \left[0,\frac 1{1-\rho}\right]$ as 
 \begin{equation} \label{def_H} H((m,\rho),y):=  -\frac{y}{2m} \left[\frac{1}{1-\rho}\log \left( \frac{1+2m}{1+2m/\rho}  \right)+ \sum_{l=1}^{\infty} \left(1- \frac{1}{(1+2m)^l}\right) \frac{(1-\rho)^{l-1}y^{l}}{l(l+1)}\right] > -\infty. \end{equation}
Several properties of the function $y \mapsto H(y):= H((m,\rho), y)$ can be derived from \cite{IGSW} and will be key in showing~\eqref{Uas}  as well as in forthcoming proofs. We collect them in the next lemma:
\begin{lemma}\label{lem_prop_UH} a) For $(m,\rho)\in\R_+\times (0,1)$,
\begin{itemize}
\item $H$ is  strictly concave function 
with $H'(1)=0$ and its  maximal value given by $H((m,\rho),1) = \eta(m,\rho)$ stated in ~\eqref{def_eta}.
\item If $\rho\leq 2/3$, $H$ is nonnegative on $[0,1/(1-\rho)]$, and if $\rho>2/3$,  $H$ is nonnegative on $[0,y_0]$ and negative on $(y_0,1/(1-\rho)]$, with $y_0=y_0(m,\rho)$  satisfying 
$$y_0\sim \frac{2}{\rho} \quad as \ m \to 0.$$
\item The first derivative of $H$ is
  \begin{equation}
 H'(y)=  -\frac{\rho}{2m(1-\rho)}\log \left( 1 + \frac{2m}{\rho}  \frac{(1-\rho)(y-1)}{(1+2m/\rho)(1-(1-\rho)y)} \right) \label{der_H}.
  \end{equation}
  \end{itemize}
  b) As $m\to 0$ the function $H$
 satisfies the following asymptotics: 
\begin{equation}\label{Hin1}
    \eta(m,\rho) = H(1) \sim \frac 1{(1-\rho)^2}\left(\frac 1\rho-1+\log \rho\right) >\frac 1{2\rho},   
\end{equation}
\begin{equation} 
H'(y) \sim \frac{1-y}{\rho(1-(1-\rho)y)}  \quad \mbox{ for all } y\in [0,1/(1-\rho)),\label{derHasym}    
\end{equation}
\begin{equation}\label{Hin2}
      H''(1) \sim-1/\rho^2,\qquad H'''(1)\sim- 2(1-\rho)/\rho^3,\quad H'''\mbox{ is bounded close to }1,
\end{equation}
\begin{equation}\label{HyH1}H(y) - H(1) + \frac{(1-y)^2}{2\rho^2} = O\left((1-y)^2m+(1-y)^3\right) \quad \mbox{as }\, y \to 1.
\end{equation}
\end{lemma}

\begin{proof}
Only \eqref{Hin2} and~\eqref{HyH1} have to be proven, for the other statements see \cite[p.~135]{IGSW}.
From
$$ H''(y)=\frac{1}{2m} \Big[ \frac{1}{1+2m-(1-\rho)y}- \frac{1}{1-(1-\rho)y} \Big] $$
(see \cite[Sec. 4.7] {IGSW}), we derive
\begin{align} \label{expr_H3} H'''(y)&=\frac{1-\rho}{2m} \Big[ \frac{1}{(1+2m-(1-\rho)y)^2}- \frac{1}{(1-(1-\rho)y)^2} \Big] \nonumber \\
&= -\frac{2(1-\rho)(1+m-(1-\rho)y)}{(1+2m-(1-\rho)y)^2(1-(1-\rho)y)^2}
. \end{align}
Hence
\begin{equation} \label{H3}  H'''(1)\sim - \frac{2(1-\rho)}{\rho^3} \quad \text{and} \quad \sup_{0 \leq y \leq 1} \big| H'''(y)\big| \leq \frac{2(1-\rho)(1+m)}{\rho^4}. \end{equation}
Now, from Taylor-Lagrange formula, we know that for $y \in (0,1)$, there exists $z \in (y,1)$ such that
\begin{align*}
    H(y)&=H(1)+ (y-1)H'(1) + \frac{(y-1)^2}{2}H''(1) + \frac{(y-1)^3}{6}H'''(z)\\
    &= H(1) + \frac{(y-1)^2}{2}\frac{1}{2m} \Big[ \frac{1}{\rho+2m}- \frac{1}{\rho}  \Big]+ \frac{(y-1)^3}{6}H'''(z) .
\end{align*}
According to the second part of \eqref{H3}, 
$$ \Big|  \frac{(y-1)^3}{6}H'''(z) \Big| \leq  \frac{2(1-\rho)(1+m)}{\rho^4}\frac{(1-y)^3}{6} .$$
From the expression of $H''$ we also get
$$ \Big| H''(1)+\frac{1}{\rho^2} \Big| = \frac{2m}{\rho^2 (\rho+2m)} \leq \frac{2m}{\rho^3}. $$
This ends the proof of the lemma.\end{proof}
\begin{proof}[Proof of Proposition~\ref{Uquad}]
 As $N\to \infty$, Equations \eqref{defb}, \eqref{link_HNH} and \eqref{HyH1} imply the asymptotics 
\begin{align*}U( \mathfrak a+K\sigma)-U(\mathfrak a)
&=  U(\mathfrak a(1+ K\sigma / \mathfrak a)) -U(\mathfrak a) 
\\&= U(\mathfrak a(1+ K\rho / \sqrt {\mathfrak u_N})) -U(\mathfrak a)
\\& = -2\mathfrak u_N (H(1+K\rho / \sqrt {\mathfrak u_N}) -H(1))
\\&=  2\mathfrak u_N (K\rho/\sqrt {\mathfrak u_N})^2 /(2\rho^2) +u_N\, O\left(\frac {K^2}{u_N} m_N+ \frac {K^3}{u_N^{3/2}}\right)  \\ &= K^2(1+O(m_N + K/\sqrt{u_N})).
\end{align*}  
\end{proof}
\normalcolor
\begin{lemma}\label{chalem}
  Let us denote by $T_\ell$ the first time at which $Y_0 = Y_0^{(N)}$ reaches the state $\ell$. 
   \begin{enumerate}[label=(\Alph*)]
       \item Let $\mathfrak d_N \leq \mathfrak a_N$ such that $\mathfrak d_N = o(\sigma_N)$. Then  
        \begin{equation*}
        \P_{\lfloor \mathfrak a_N - \mathfrak d_N \rfloor} \left(T_0<T_{\mathfrak a_N}\right)\sim  \frac 2{\mathfrak c_N} \mathfrak d_N e^{-2\mathfrak u_N H(1)} \quad \mbox{as } N\to \infty.
    \end{equation*} \label{hitprobgenenum}
    \item Let $1 \ll \mathfrak d_N<\mathfrak a_N$, $ \mathfrak f_N>2$. Then for $N$ large enough,
    \begin{align*} 
\mathbf P_{\left\lfloor \mathfrak a_N-\mathfrak d_N\right\rfloor}(T_{\mathfrak a_N-\mathfrak f_N  \mathfrak d_N }< T_{\mathfrak a_N}) & \leq e^{-(\mathfrak f_N-2) \mathfrak d_N^2\mathfrak{u}_N /\mathfrak a_N^2}.
\end{align*}\label{maj_proba_hittingenum}
   \end{enumerate}
\end{lemma}
\begin{proof}
We begin with the proof of \ref{hitprobgenenum}.
As already observed the function $R=R^{(N)}$ defined in \eqref{RKdef} is harmonic for the time-discrete birth-and-death chain associated with $Y_0$, hence
    \begin{equation}
        \P_{\lfloor \mathfrak a - \mathfrak d \rfloor} \left(T_0<T_{\mathfrak a}\right)= \frac{ R_{\mathfrak{a}}- R_{\mathfrak{a}-\mathfrak d}}{R_{\lfloor \mathfrak{a}\rfloor }}.
    \end{equation}
    Because of~\eqref{defU},~\eqref{link_HNH} and Lemma~\ref{lem_prop_UH}, and as $ \mathfrak d_N = o(N/m)$ we have for every $\mathfrak a_N - \mathfrak d_N \leq n \leq \mathfrak a_N$,
$$r_{n}\sim e^{-2\mathfrak u_N H(1)}.$$
    According to \cite[Lemma 4.6]{IGSW}, 
$$R_{\lfloor \mathfrak a_N\rfloor} \sim  \frac{\rho_N}{2m_N(1-\rho_N)}  = \frac { \mathfrak c_N }2. $$
This concludes the proof of \ref{hitprobgenenum}. \\
Now let $1 \ll \mathfrak d_N<\mathfrak a_N$, $ \mathfrak f_N>2$. We have
\begin{align*} 
\mathbf P_{\left\lfloor \mathfrak a-\mathfrak d\right\rfloor}(T_{\mathfrak a-\mathfrak f  \mathfrak d }< T_{\mathfrak a}) &= \frac{R_{\mathfrak{a}}- R_{\mathfrak{a}-\mathfrak d}}{R_{\mathfrak{a}}- R_{\mathfrak{a}-\mathfrak f  \mathfrak d }}\\
\nonumber& \sim \frac{\int_{1-\mathfrak d/\mathfrak a}^1e^{-2\mathfrak{u}_NH(z)}dz}{\int_{1-\mathfrak f  \mathfrak d/\mathfrak a}^1e^{-2\mathfrak{u}_NH(z)}dz}\\
\nonumber&  \leq  \frac{\int_{1- \mathfrak d/ \mathfrak a}^1e^{-2\mathfrak{u}_NH(z)}dz}{\int_{1-\mathfrak f  \mathfrak d/\mathfrak{a}}^{1-(\mathfrak f-1)  \mathfrak d/\mathfrak{a}}e^{-2\mathfrak{u}_NH(z)}dz}\\
\nonumber& \leq e^{-2\mathfrak{u}_N(H(1-\mathfrak d/\mathfrak a)-H(1-(\mathfrak f-1)\mathfrak d /\mathfrak a)}\\\nonumber
& \leq e^{-2(\mathfrak f-2) (\mathfrak d/\mathfrak a)\mathfrak{u}_N H'(1-\mathfrak d/\mathfrak a)},
\end{align*}
where we used that $H$ is non-decreasing and concave on $[0,1]$. From \eqref{der_H} we see that for any $y \in (0,1)$,
\begin{align*}
 H'(y)&\geq\frac{\rho}{2m(1-\rho)}\frac{2m}{\rho}  \frac{(1-\rho)(1-y)}{(1+2m/\rho)(1-(1-\rho)y)} \\
 &=   \frac{1-y}{(1+2m/\rho)(1-(1-\rho)y)} \geq   \frac{1-y}{1+2m/\rho} .
\end{align*}
We deduce that for $m$ small enough
\begin{align*} 
\mathbf P_{\left\lfloor \mathfrak a-\mathfrak d\right\rfloor}(T_{\mathfrak a-\mathfrak f  \mathfrak d }< T_{\mathfrak a}) & \leq e^{-2(\mathfrak f-2) \mathfrak d^2\mathfrak{u}_N /( (1+2m/\rho) \mathfrak a^2)}\leq e^{-(\mathfrak f-2) \mathfrak d^2\mathfrak{u}_N /\mathfrak a^2},
\end{align*}
which ends the proof of \ref{maj_proba_hittingenum}.
\end{proof}
\begin{remark}\label{altern} In a similar way as in the proof of Lemma~\ref{chalem} one can show, using~\eqref{RKdef} and Lemma~\ref{lem_prop_UH}, that for all $K >0$
\begin{equation}\label{hitK}
        \P_{\lfloor \mathfrak a_N - K\sigma_N \rfloor} \left(T_0<T_{\mathfrak a_N}\right)\sim \frac 1{e_N} \cdot c_N \cdot 2\sqrt{\pi} \int_0^Ke^{y^2} {\rm d}y \quad \mbox{as } N\to \infty.
    \end{equation} 
    The l.h.s. of~\eqref{hitK} is the inverse of the expected number of excursions of $Y_0^{(N)}$ from $\mathfrak a_N$ that reach $\lfloor \mathfrak a_N - K\sigma_N \rfloor$, and the last of the three factors on the r.h.s ~\eqref{hitK} is the expected time it takes the process $\mathcal H$ defined in~\eqref{OUSDE} to visit $-K$ and then travel back to 0, cf. \cite[Lemma 3]{grubel2005rarity}. In the light of the Ornstein-Uhlenbeck approximation that will be proved in Sec.~\ref{OUfluc}, the asymptotics \eqref{hitK} not only provides a factorisation of~$e_N$ that is similarly intuitive as~\eqref{eNnew}, but also opens a route for an application of~\cite[Theorem 1]{grubel2005rarity} in order to prove Theorem~\ref{thmY0}.\ref{thmY0:d}. As already mentioned in Remark~\ref{Rem2_5}, in the next subsections we take a different route that leads via quasi-equilibria. 
 \end{remark}
\subsection{Expected hitting times: Proof of Theorem~\ref{thmY0}.\ref{thmY0:a}}
\begin{lemma}\label{comingdown}
Under assumptions~\eqref{smallmut} and \eqref{expreg}, for $N$ large enough, 
\begin{equation}\label{comedown}
 \mathbf{E}_N[T_\mathfrak{a}]\leq  (1 \vee 4 \rho)\rho\, \mathfrak c_N  \log \mathfrak{a}.   
\end{equation}  \end{lemma}

\begin{proof}
 We are interested in the expectation of the first hitting time of $\mathfrak{a}$ when the initial state is~$N$. According to Equation (4.5) in \cite{IGSW}, if we denote by $T_j=T_j^{(N)}$ the first hitting time of $j \in \mathbb{N}$ we obtain
\begin{align*}
\mathbf E_N[T_\mathfrak{a}]= \mathbf E_N[T_0]- \mathbf E_\mathfrak{a}[T_0]
& = \sum_{n=1}^{N-1} \frac{R_{n\wedge N}-R_{n \wedge \mathfrak{a}}}{\lambda_n r_n}+ \frac{R_N-R_\mathfrak{a}}{\mu_N r_{N-1}}\\
& = \sum_{n=\mathfrak{a}+1}^{N-1} \frac{R_{n}-R_{\mathfrak{a}}}{\lambda_n r_n}+ \frac{R_N-R_\mathfrak{a}}{\mu_N r_{N-1}}.
\end{align*}
1) Let us first consider integers $n$ such that $n \leq C\mathfrak{a}$ with $C>2/\rho$. 
Following the proof of Lemma~4.6 in \cite{IGSW}, we obtain for such $n$,
\begin{align*}
\frac{R_{n}-R_{\mathfrak{a}}}{\lambda_n r_n} &\sim \frac{2}{n(1-n/N)}\int_\mathfrak{a}^n e^{-2{u}_N(H(y/\mathfrak{a})-H(n/\mathfrak{a})}dy\\
&= \frac{2N(1-\rho)}{n(1-n/N)}\int_1^{n/\mathfrak{a}} e^{-2{u}_N(H(z)-H(n/\mathfrak{a})}dz\\
&\leq \frac{2N(1-\rho)}{n(1-n/N)}\int_1^{n/\mathfrak{a}} e^{2{u}_N H'(z)(n/\mathfrak{a}-z)}dz\\
\end{align*}
where we have used that $H$ is concave (see \cite[p. 135]{IGSW} and Lemma \ref{lem_prop_UH}).
For $1 < z \leq C$, we obtain from \eqref{derHasym} that
$$ H'(z)\sim - \frac{z-1}{\rho (1-(1-\rho)z)}\leq - \frac{z-1}{\rho^2}. $$
We deduce that
\begin{align*}
\frac{R_{n}-R_{\mathfrak{a}}}{\lambda_n r_n} &\leq \frac{2N(1-\rho)}{n(1-n/N)}\int_1^{n/\mathfrak{a}} e^{2{u}_N (z-1)(n/\mathfrak{a}-z)/\rho^2}dz\\
 &= \frac{2N(1-\rho)}{n(1-n/N)} e^{-2{u}_N ((n/\mathfrak{a}-1)/2)^2/\rho^2} \int_{-(n/\mathfrak{a}-1)/2}^{(n/\mathfrak{a}-1)/2} e^{2{u}_N u^2/\rho^2}dz  \\
  &= \frac{2N(1-\rho)}{n(1-n/N)} e^{-2{u}_N ((n/\mathfrak{a}-1)/2)^2/\rho^2}\frac{\rho}{\sqrt{2{u}_N}}2 \int_0^{(\sqrt{2{u}_N}/\rho)(n/\mathfrak{a}-1)/2} e^{u^2}dz .
\end{align*}
Using that for any $A>0$, 
$$\int_0^A e^{x^2}dx \leq 2\int_{A/2}^A \leq \frac{2}{2(A/2)}\int_{A/2}^A 2xe^{x^2}dx \leq \frac{2}{A}e^{A^2}, $$
and that $1-n/N \sim 1$,
we get
\begin{align*}
\frac{R_{n}-R_{\mathfrak{a}}}{\lambda_n r_n} &\leq 4 \rho^2 \mathfrak c_N \frac{\mathfrak{a}}{n(n-\mathfrak{a})}.
\end{align*}
2) Let us now consider $n \in \mathbb{N}$ satisfying $2\mathfrak{a}/\rho \leq n \leq N(1-\sqrt{m})$. Then according to Lemma 4.6 in \cite{IGSW},
$$ R_n \sim \frac{\rho(1-n/N)\mathfrak c_N}{2(n/\mathfrak{a}-1)} e^{-2u_N H(n/\mathfrak{a})}.$$
This implies that
\begin{align*}
\frac{R_{n}-R_{\mathfrak{a}}}{\lambda_n r_n} \sim \frac{2}{n(1-n/N)}\frac{\rho(1-n/N)}{2 (n/\mathfrak{a}-1)}\mathfrak c_N= \rho \mathfrak c_N \frac{\mathfrak{a}}{n(n-\mathfrak{a})}
\end{align*}
3) Let us now consider $n \in \mathbb{N}$ satisfying $N(1-\sqrt{m})\leq n \leq N-1$. Then according to \eqref{der_H} (or Equation (4.37) in \cite{IGSW}), 
\begin{align*}
H'(n/\mathfrak{a})&= -\frac{\mathfrak c_N}{2} \log \Big( 1+ \frac{2m}{\rho} \frac{(1-\rho)(n/\mathfrak{a}-1)}{(1+2m/\rho)(1-n/N)}\Big)\\
&\sim  -\frac{\mathfrak c_N}{2} \log \Big( 1+ \frac{2m}{1-n/N}\Big)=-\frac{\mathfrak c_N}{2} \log \Big( 1+ \frac{2mN}{N-n}\Big)
\end{align*}
Hence
$$ |H'(n/\mathfrak{a})| \geq \frac{\mathfrak c_N}{2} \log \Big( 1+ \frac{2mN}{N\sqrt{m}}\Big) \sim \sqrt{m}\mathfrak c_N. $$
In particular, 
\begin{align*}
\frac{R_{n}-R_{\mathfrak{a}}}{r_n} &\leq N(1-\rho)\int_1^{n/\mathfrak{a}} e^{2u_N(n/\mathfrak{a}-z)/(\sqrt{m}(1-\rho))}dz\leq \frac{1}{2\sqrt{m}}
\end{align*}
and
\begin{align*}
\frac{R_{n}-R_{\mathfrak{a}}}{\lambda_nr_n} &\leq  \frac{1}{\sqrt{m}(N-n)}.
\end{align*}
4) Finally, for $n=N$, as the $r_k$'s are increasing with $k\geq \mathfrak{a}$, 
$$ \frac{R_N-R_\mathfrak{a}}{\mu_N r_{N-1}}\leq \frac{(N-\mathfrak{a})r_{N-1}}{mN r_{N-1}} \leq \frac{1}{m}.$$
Combining points 1) to 4) we may now conclude the proof:
\begin{align*}
\mathbf{E}_N[T_\mathfrak{a}]& = \sum_{n=\mathfrak{a}+1}^{N-1} \frac{R_{n}-R_{\mathfrak{a}}}{\lambda_n r_n}+ \frac{R_N-R_\mathfrak{a}}{\mu_N r_{N-1}}\\
& = \sum_{n=\mathfrak{a}+1}^{3 \mathfrak{a}/\rho} \frac{R_{n}-R_{\mathfrak{a}}}{\lambda_n r_n}+\sum_{n=3\mathfrak{a}/\rho+1}^{N(1-\sqrt{m})} \frac{R_{n}-R_{\mathfrak{a}}}{\lambda_n r_n}+\sum_{n=N(1-\sqrt{m})+1}^{N-1} \frac{R_{n}-R_{\mathfrak{a}}}{\lambda_n r_n}+ \frac{R_N-R_\mathfrak{a}}{\mu_N r_{N-1}}\\
\\
& \leq \sum_{n=\mathfrak{a}+1}^{N(1-\sqrt{m})}  \rho \mathfrak c_N(1 \vee 4 \rho) \frac{\mathfrak{a}}{n(n-\mathfrak{a})} +\sum_{n=N(1-\sqrt{m})+1}^{N-1} \frac{1}{\sqrt{m}N}+ \frac{1}{m}\\
& \leq  \rho \mathfrak c_N(1 \vee 4 \rho) \sum_{n=\mathfrak{a}+1}^{N(1-\sqrt{m})}\Big(\frac{1}{n-\mathfrak{a}}-\frac{1}{n}\Big) +1+ \frac{1}{m}\sim  (1 \vee 4 \rho) \rho \mathfrak c_N\log \mathfrak{a}.
\end{align*}
This ends the proof.
\end{proof}
 We will now consider a modification $Y_*^{(N)}$ of $Y_0^{(N)}$ which has the same dynamics as $Y_0^{(N)}$ except in state $0$, where, instead of being absorbed, it is ``softly reflected''. The equilibrium distribution of $Y_*^{(N)}$ will then serve  as a proxy for the quasi-equilibrium distribution of $Y_0^{(N)}$ as $N\to \infty$.
\normalcolor
\begin{definition}\label{Ystar}Let $Y_*=Y_*^{(N)}$ be the
 $\{0,1,\ldots,N\}$-valued Markov chain  whose jump rates $\lambda_n^*$ and $\mu_n^*$ are the same as the jump rates $\lambda_n$  and $\mu_n$ of $Y_0^{(N)}$ except in the  state $0$, where we put $\lambda^*_0  =1$ and $\mu^*_0  =0$.  We write $\pi_*^{(N)}(n)$ for the weight of the equilibrium distribution of $Y_*^{(N)}$ in $n$.
\end{definition}
\begin{lemma}\label{lemtimetoa}
    With  $T_\mathfrak a^*$ defined as  the first time when  $Y_*^{(N)}$ visits  $\lfloor \mathfrak a_N \rfloor$, under assumptions~\eqref{smallmut} and \eqref{expreg} one has
    \begin{equation}\label{timetoa}
       \E_0[T_\mathfrak a^*]:= \E[T_\mathfrak a^*\mid Y_*^{(N)}(0) =0 ]= O(\mathfrak c \log \mathfrak a).  
    \end{equation}
\end{lemma}
\begin{proof} 
Proceeding similarly as in~\cite[Section 4]{IGSW} we obtain (omitting the floor brackets around~$\mathfrak a$) with $r_k$ as in~\eqref{def_ri}
\begin{align}\label{reach_time} \nonumber \E_0[T_\mathfrak a^*]- \sum_{k=0}^{ \mathfrak a -1} r_k &= \sum _{n=1}^{\mathfrak{a}-1}\frac{1}{\lambda_n} \sum_{k=n}^{\mathfrak{a}-1}\frac{r_k}{r_n}\\
\nonumber &=\frac{2}{1/2+m/\rho} \sum _{n=1}^{\mathfrak{a}-1}\frac{1}{n(1-(1-\rho)(n/\mathfrak{a}))} \sum_{k=n}^{\mathfrak{a}-1}e^{-2u_N (H(k/\mathfrak{a})-H(n/\mathfrak{a}))} \\
& \sim 4\sum _{n=1}^{\mathfrak{a}-1}\frac{1}{n(1-(1-\rho)(n/\mathfrak{a}))}\int_n^{\mathfrak{a}-1 } e^{-2u_N (H(y/\mathfrak{a})-H(n/\mathfrak{a}))}dy. \end{align}
We will decompose the last sum in \eqref{reach_time} into two parts. First notice that if $\mathfrak a-\sigma \leq n \leq y \leq \mathfrak a$, 
$$ \frac{1}{n(1-(1-\rho)(n/\mathfrak{a}))} \sim \frac{1}{\rho \mathfrak a} .$$
Hence, as $H$ is non-decreasing from $0$ to $1$, 
\begin{eqnarray}   &&\sum _{n=\mathfrak a - \sigma}^{\mathfrak{a}-1}\frac{1}{n(1-(1-\rho)(n/\mathfrak{a}))}\int_n^{\mathfrak{a}-1 } e^{-2u_N (H(y/\mathfrak{a})-H(n/\mathfrak{a}))}dy \nonumber\\&\leq& \frac{1}{\rho \mathfrak a} \sum _{n=\mathfrak a - \sigma}^{\mathfrak{a}-1} (\mathfrak a -n) \nonumber \\
&=& \frac{1}{\rho \mathfrak a}\frac{\sigma(\sigma +1)}{2} \nonumber \\
& \sim & \frac{1}{\rho \mathfrak a}\frac{\rho^2 N}{2m} = \frac{\rho\mathfrak c}{2}. \label{reach1} \end{eqnarray} 
Recall that $H''\leq 0$. In particular, this implies that for any $n \leq y \leq  \mathfrak{a}-1$,
$$ H(y/\mathfrak{a})\geq H(n/\mathfrak{a}) + \frac{y-n}{\mathfrak{a}}H'(y/{\mathfrak{a}}) .$$
and according to \eqref{derHasym}
$$ 2u_N (H(y/\mathfrak{a})-H(n/\mathfrak{a}))\geq 2u_N  \frac{y-n}{\mathfrak{a}}H'(y/{\mathfrak{a}})\sim 2u_N  \frac{y-n}{\mathfrak{a}} \frac{\mathfrak a-y}{\rho \mathfrak a}$$
 We deduce that
 \begin{align*}   \int_n^{\mathfrak{a}} e^{-2u_N (H(y/\mathfrak{a})-H(n/\mathfrak{a}))}dy&\leq  \int_n^{\mathfrak{a}} e^{- \frac{2u_N}{\rho \mathfrak a^2} (y-n)(\mathfrak a-y)}dy\\
& =  \int_{-(\mathfrak{a}-n)/2}^{(\mathfrak{a}-n)/2} e^{- \frac{2u_N}{\rho \mathfrak a^2}((\mathfrak{a}-n)^2/4-z^2)}dz\\
&= \mathfrak a \sqrt{\frac{2\rho}{u_N}}  e^{- \frac{2u_N}{\rho \mathfrak a^2}(\mathfrak{a}-n)^2/4}\int_{0}^{\sqrt{2u_N/\rho \mathfrak a^2} (\mathfrak{a}-n)/2} e^{x^2}dx. \end{align*}
Let us introduce the Dawson function, for $x \geq 0$ 
$$ F(x):= e^{-x^2}\int_0^x e^{t^2}dt.$$
Then it is known  that $F(x) \sim 1/2x$ for large $x$ (see for instance \cite[(1), (9)]{wolframDwason}). 
We have just proven that
 \begin{equation*}   \int_n^{\mathfrak{a}} e^{-2u_N (H(y/\mathfrak{a})-H(n/\mathfrak{a}))}dy \leq 
 \mathfrak a \sqrt{\frac{2\rho}{u_N}} F\left( \frac{\mathfrak a-n}{\mathfrak a}\sqrt{\frac{u_N}{2\rho}} \right). \end{equation*} 
 But $\mathfrak a -n \geq \sigma $ implies
 $$  \frac{\mathfrak a-n}{\mathfrak a}\sqrt{\frac{u_N}{2\rho}}\geq  \frac{\sigma}{\mathfrak a}\sqrt{\frac{u_N}{2\rho}} = \frac{\rho \sqrt{N}}{\sqrt{m}N(1-\rho)} \sqrt{\frac{mN(1-\rho)^2}{2\rho}} = \sqrt{\frac{\rho}{2}}. $$
As this last term is lower bounded by $\sqrt{\rho^\ast}/2$ for $N$ large enough, we deduce that there exists a positive constant $C$ such that for $N$ large enough and $\mathfrak a -n \geq \sigma $, 
$$ F\Big( \frac{\mathfrak a-n}{\mathfrak a}\sqrt{\frac{u_N}{2\rho}} \Big) \leq \frac{C}{2} \frac{\mathfrak a}{(\mathfrak a-n)\sqrt{u_N}} $$
and thus
$$   \int_n^{\mathfrak{a}} e^{-2u_N (H(y/\mathfrak{a})-H(n/\mathfrak{a}))}dy \leq C \frac{\mathfrak a^2}{(\mathfrak a-n){u_N}}. $$
We thus obtain
\begin{eqnarray*}
    &&\sum _{n=1}^{\mathfrak{a}-\sigma}\frac{1}{n(1-(1-\rho)(n/\mathfrak{a}))}\int_n^{\mathfrak{a}-1 } e^{-2u_N (H(y/\mathfrak{a})-H(n/\mathfrak{a}))}dy \\&\leq&\frac{C\mathfrak a}{\rho{u_N}}   \sum _{n=1}^{\mathfrak{a}-\sigma}\frac{\mathfrak{a}}{n(\mathfrak{a}-n)}\\
    &=& \frac{C\mathfrak c}{\rho^2}  \sum _{n=1}^{\mathfrak{a}-\sigma}\Big(\frac{1}{n}+ \frac{1}{\mathfrak{a}-n}\Big) \sim \frac{2C\mathfrak c}{\rho^2} \log \mathfrak{a}.
\end{eqnarray*}
Combining this with \eqref{reach1}, we see that the r.h.s. of~\eqref{reach_time} is $O(\mathfrak c \log \mathfrak a)$. This readily imples the assertion of the lemma, since $\sum_{k=0}^{ \mathfrak a -1} r_k=O(\mathfrak c)$ (see  \cite[Lemma 4.6]{IGSW}).
\end{proof}
Next we observe that
\begin{equation}\label{concatex}
   \E[T_{\mathfrak a}\, |\,  Y_0^{(N)}(0) = 1\mbox{ and } T_0 > T_{\mathfrak a}] \le \E[T_\mathfrak a^*\mid Y_*^{(N)}(0) =0 ]. 
\end{equation}
Indeed, decomposing $(Y_*(t\wedge T^*_{\mathfrak a}))_{t\ge 0}$ into its excursions from $0$ that remain below $\lfloor \mathfrak a \rfloor$,  and the  piece that  goes from $0$ to $\lfloor \mathfrak a\rfloor$ without ever returning to $0$, we see that the latter, when observed from the time at which it jumps from $0$ to $1$, has (up to a time shift) the same distribution as $(Y(t\wedge T_{\mathfrak a}))_{t\ge 0}$ under $\P_1(\cdot \mid T_{\mathfrak a} < T_0)$. This proves the inequality~\eqref{concatex}.  
Combining~\eqref{concatex} with Lemma~\ref{lemtimetoa} we see (always
under assumptions~\eqref{smallmut} and~\eqref{expreg}) that
\begin{equation}\label{dursweep}
\E[T_{\mathfrak a} \, | \, Y_0^{(N)}(0) = 1, T_0 > T_{\mathfrak a}] = O(\mathfrak c \log \mathfrak a).    
\end{equation} 

\noindent To conclude the proof of Theorem~\ref{thmY0}.\ref{thmY0:a}, because of~\eqref{comedown} and~\eqref{dursweep}  it only remains to show
\begin{equation}\label{durlasttravel}
\E\left[T_0  \,\mid \, Y_0^{(N)}(0) = \lfloor\mathfrak a\rfloor - 1, \, T_0 < T_\mathfrak a\right]=O(\mathfrak c\log \mathfrak a).   
\end{equation}
This, however, is a direct consequence of~\eqref{dursweep} combined with the next lemma. 
\begin{lemma}\label{timereveralllemma}
    The number of steps of the last excursion of $Y_0$ from $\lfloor \mathfrak a \rfloor$ has (when decreased by $1$) the same distribution as the number of steps it takes $Y_0$ to reach $\lfloor \mathfrak a \rfloor$ when started in $1$ and being conditioned on the event  $\{T_{\mathfrak a}< T_0\}$.
\end{lemma}
\begin{proof}We consider the discrete-time Markov chain $V$ on $\{0,1,\ldots, \lfloor a \rfloor, \lfloor a \rfloor +1\} $ with transition probabilities 
$$P(n,n+1):=\frac {\lambda_n}{\lambda_n+\mu_n} \quad \text{and} \quad P(n,n-1):= 1-P(n,n+1) \quad \text{for} \quad n=1,\ldots, \lfloor \mathfrak a \rfloor,$$
and 
$$P(0,1) = P(\lfloor \mathfrak a \rfloor +1, \lfloor \mathfrak a \rfloor ):= 1.$$
The chain $V$ has a reversible equilibrium distribution which assigns non-zero weights to both $0$ and~$\lfloor a \rfloor$. Let $V_0$ (resp. $V_{\lfloor \mathfrak a \rfloor}$) be the Markov chain with transition probability $P$ starting in $0$ (resp. in~$\lfloor \mathfrak a \rfloor$).  Say that an excursion of $V_{\lfloor \mathfrak a \rfloor}$ from $\lfloor \mathfrak a \rfloor$ is successful if it ever reaches $0$, and that an excursion of $V_0$ from $0$ is successful if it ever reaches $\lfloor \mathfrak a \rfloor$. 
Let  $\widetilde T_{\lfloor \mathfrak a \rfloor \to 0}$ be the random number of steps which the first successful excursion of $V_{\lfloor \mathfrak a \rfloor}$ from $\lfloor \mathfrak a \rfloor$ needs to reach $0$, and let    $\widetilde T_{0\to \lfloor \mathfrak a \rfloor}$ be the number of steps which the first successful excursion of $V_0$ needs to reach $\lfloor \mathfrak a \rfloor$. 
   We conclude from  \cite[Proposition 4]{barrera2009abrupt} that $\widetilde T_{\lfloor \mathfrak a \rfloor \to 0}$ has the same distribution as $\widetilde T_{0\to \lfloor \mathfrak a \rfloor}$.  The claim follows by noting that $\widetilde T_{\lfloor \mathfrak a \rfloor \to 0}$ has the same distribution as the number of steps of the last excursion of $Y_0$ from $\lfloor \mathfrak a \rfloor$, and that  $\widetilde T_{0\to \lfloor \mathfrak a \rfloor}-1$ has the same distribution as the number of steps it takes $Y_0$ to reach $\lfloor \mathfrak a \rfloor$ when started in $1$ and being conditioned on the event  $\{T_{\mathfrak a}< T_0\}$.  
\end{proof}
\subsection{Ornstein-Uhlenbeck fluctuations: Proof of Theorem~\ref{thmY0}.\ref{thmY0:b}}\label{OUfluc} 
 We will prove the ``convergence of generators'' of the processes $\mathcal H_N$ (defined in~\eqref{OUconvdef}) to the generator of the Ornstein Uhlenbeck processes $\mathcal H$, and check a compact containment condition. A reference in which results contained in~\cite{ethier2009markov} are nicely refined in a framework suitable for our purpose is~\cite[Proposition 2.5]{BW2025}.  In view of the definitions~\eqref{def_a}, ~\eqref{defb} and~\eqref{defh}, the process $\mathcal H_N$  jumps from $x$ to $x+\frac 1{\sigma_N}=x+\sqrt{\tfrac s{\rho N}}$ at rate
$$ \beta_N(x)= \frac{\rho}{m(1-\rho)}\Big(\sqrt{\tfrac {\rho N }{ s}}x +N(1-\rho) \Big)\Big( \frac{1}{2}+ \frac{m}{\rho} \Big)\Big( \rho- \frac{x\sqrt  \rho}{\sqrt{sN}} \Big) $$ 
and from $x$ to $x-\sqrt{\tfrac s{\rho N}}$ at rate
$$ \delta_N(x)= \frac{\rho}{m(1-\rho)}\Big( \sqrt{\tfrac {\rho N} { s }}x +N(1-\rho) \Big)\Big( \frac{1}{2}\Big( \rho- \frac{x{\sqrt  \rho}}{\sqrt{s N}} \Big)+ m \Big) .$$
In particular, 
\begin{equation} \label{beta-delta} \beta_N(x)-\delta_N(x) = -\frac{1}{1-\rho}\Big( \sqrt{\tfrac {\rho N} { s }} x +N(1-\rho) \Big) \frac{x\sqrt \rho}{\sqrt{sN}}  \end{equation}
and 
\begin{equation} \label{beta+delta}\beta_N(x)+\delta_N(x) = \frac{\rho}{m(1-\rho)}\Big( \sqrt{\tfrac {{\rho}N} { s}}x +N(1-\rho) \Big)\Big(\rho + 2m - \frac{x{\sqrt \rho}}{\sqrt{sN}}\Big(1+ \frac{m}{\rho}\Big) \Big). \end{equation}
Applied to a function $f$ with three bounded continuous derivatives, the generator of $\mathcal H_N$ thus takes the form
\begin{align*} G_N f(x)&=\beta_N(x)\Big( f\Big(x+\sqrt{\tfrac s{{\rho}N}}\Big)-f(x) \Big)+\delta_N(x)\Big( f\Big(x-\sqrt{\tfrac s{{\rho}N}}\Big)-f(x) \Big) \\
&= \sqrt{\frac{s}{{ \rho} N}}(\beta_N(x)-\delta_N(x))f'(x)+\frac{s}{2N{ \rho}}(\beta_N(x)+\delta_N(x))f''(x)+O\Big( \frac{s}{N} \Big)^{3/2}.\end{align*}
Using \eqref{beta-delta} one obtains
$$  \sqrt{\tfrac s{{\rho}N}}(\beta_N(x)-\delta_N(x))= - \Big( 1+ \frac{\rho}{\sqrt{ \mathfrak u_N  }}x \Big)x, $$
and \eqref{beta+delta} yields
\begin{equation}\label{quadvar}
    \frac{s}{2N{\rho}}(\beta_N(x)+\delta_N(x))= \frac{1}{2{\rho}}\Big( 1+ \frac{\rho}{\sqrt{u_N}}x \Big)\Big(\rho + 2m - \frac{x\rho}{\sqrt{mN}}\Big(1+ \frac{m}{\rho}\Big) \Big). 
\end{equation}
We conclude that
$$  G_N f(x)=\Big( 1+  \frac{\rho}{\sqrt{u_N}}x \Big)\Big[-xf'(x)+\frac{1}{2}\Big(1 + \frac{2m}{\rho} - \frac{x}{\sqrt{mN}}\Big(1+ \frac{m}{\rho}\Big) \Big)f''(x)\Big]+ O\Big( \frac{m}{N} \Big)^{3/2}.$$
In particular, the ``generator convergence'' condition 3 of~\cite[Proposition 2.5]{BW2025} is satisfied with 
$$ Gf(x)= -xf'(x)+  \frac{f''(x)}{2}, $$ 
 which is the generator of the standard Ornstein-Uhlenbeck process.  
Now we will prove that the \emph{compact containment condition} (condition 4 in~\cite[Proposition 2.5]{BW2025})  is fullfilled as well. \normalcolor For this we have to show that for any given (arbitrarily large) $T>0$ and (arbitrarily small) $\varepsilon >0$ there exists $K>0$ such that 
\begin{equation}\label{compcont}
\mathbf P\left(\sup_{0\le t\le T} \mid \mathcal H_N(t) \mid \, \ge K\right) < \varepsilon.
\end{equation} 
We first prove that there exists a natural number $n$ (independent of $N$) such that
\begin{equation}\label{fewexcursions}
    \P(\mathcal H_N \mbox{ travels between times } 0 \mbox{ and } T \mbox{ more than } n \mbox{ times from } 0 \mbox{ to } -1) < \frac \varepsilon 2.
\end{equation}
To this end we denote by $\widehat {\mathcal H}_N$ the process that arises by reflecting $\mathcal H_N$ from below in $0$, and by  $\mathcal H_N^\ast$ a ``driftless version'' of  $\widehat {\mathcal H}_N$.  More precisely, for $x<0$ the jump rates $\widehat \beta_N(x)$ and $\widehat \delta_N(x)$ of $\widehat {\mathcal H}_N$ are the same as those of $\mathcal H_N$, while  $\widehat \beta_N(0) =0$ and  $\widehat \delta_N(0) =\beta_N(0)+\delta_N(0)$. The jump rates $\beta^\ast_N(x)$ and $\delta^\ast_N(x)$ of $\mathcal H_N^\ast$ are the same as those of $\widehat {\mathcal H}_N$ for $x=0$, and are
$$\beta^\ast_N(x)=\delta^\ast_N(x) =(\widehat \beta_N(x)+\widehat \delta_N(x))/2 $$
for $x< 0$. Let  $\zeta_N$ (resp. $\widehat \zeta_N$ resp. $\zeta_N^\ast$) be the first time at which $\mathcal H_N$ (resp.~$\widehat {\mathcal H}_N$ resp. $\mathcal H_N^\ast$)  hits $-1$ when started in $0$. By construction these three stopping times are stochastically ordered in the sense that for all $t>0$
\begin{equation}\label{stochmon1}
    \P(\zeta_N>t) \ge \P(\widehat \zeta_N>t)\ge \P(\zeta_N^\ast>t).
\end{equation} 
To bound this from below, let $\tau_N^\ast$ be the time at which $\mathcal H_N^\ast$ first hits $\{0,-1\}$ when started in $-1/2$. We observe that
\begin{equation}\label{stochmon}
  \P(\zeta_N^\ast\le t) \le \P(\tau_N^\ast\le t) =\P_{-\frac 12}\left(\sup_{0\le u\le t}|\mathcal H_N^\ast(u)-(-\tfrac 12)| \ge \tfrac 12\right)
\end{equation}

According to \eqref{beta+delta} we have for $x<0$,
\begin{align*}
    \frac{\beta^\ast_N(x)+\delta^\ast_N(x)}{\sigma_N^2}& = \frac{\beta_N(x)+\delta_N(x)}{\sigma_N^2}\\
 &=    \frac{1}{N\rho(1-\rho)}\Big( \sqrt{\tfrac {{\rho}N} { s}}x +N(1-\rho) \Big)\Big(\rho + 2m - \frac{x{\sqrt \rho}}{\sqrt{sN}}\Big(1+ \frac{m}{\rho}\Big) \Big)\\
 &=   \Big( 1+x\frac{\rho}{\sqrt{u_N}}\Big)\Big(1 + \frac{2m}{\rho} - \frac{x}{\sqrt{mN}}\Big(1+ \frac{m}{\rho}\Big) \Big)
\end{align*}
Hence for any $x \in [-1,0)$,
$$ \frac{\beta^\ast_N(x)+\delta^\ast_N(x)}{\sigma_N^2}\leq  1 + \frac{2m}{\rho} + \frac{1}{\sqrt{mN}}\Big(1+ \frac{m}{\rho}\Big)\to 1, \quad \mbox{ as } N \to \infty. $$
Thus for $N$ large enough,
$$(\beta^\ast_N(x)+\delta^\ast_N(x))\frac 1{\sigma_N^2} \le 2.$$
Consequently,  the quadratic variation of the martingale ${\mathcal H}_N^\ast(t\wedge \tau_N^\ast)$, $t\ge 0$ is bounded from above by $2t$. Thus, by Doob's $L^2$ inequality the r.h.s. of~\eqref{stochmon} is bounded from above by $2t$. Combining this with~\eqref{stochmon1} and~\eqref{stochmon} we obtain
\begin{equation}\label{largenough}
    \P(\zeta_N > 1/4) \ge \frac 12.
\end{equation}
With $\zeta_N^{(1)}, \zeta_N^{(2)}, \ldots $ independent copies of $\zeta_N$, we conclude from~\eqref{largenough} that there exists a natural number  $n$ (not depending on $N$) for which
$$\P(\zeta_N^{(1)} +\cdots + \zeta_N^{(n)}\le T) < \varepsilon/2.$$
 By construction of $\zeta_N$, \eqref{fewexcursions} is satisfied for this $n$. 
 
Applying   Lemma~\ref{chalem}\ref{maj_proba_hittingenum} we get that for any $K>2$,
$$ \mathbf P_{\left\lfloor \mathfrak a-\sigma\right\rfloor}(T_{\mathfrak a-K\sigma}< T_{\mathfrak a}) \leq e^{-(K-2)}. $$
While this estimate deals with the large excursions of $\mathcal H_N$ {\em below} $0$, an analogous estimate is obtained for the excursions of $\mathcal H_N$ {\em above} $0$ by noting that  the reverting drift of $Y_0$ above~$\mathfrak a$  is  not smaller than that below~$\mathfrak a$.  In order to satisfy~\eqref{compcont} it thus suffices to choose $K=K(\varepsilon, T)$ such that $n e^{-(K-2)} < \varepsilon/2$. 
\subsection{Asymptotic normality of the Boltzmann-Gibbs distribution}\label{BGdist}
Let $Y_*^{(N)}$ and its equilibrium distribution $\pi_*^{(N)}$ be as in Definition~\ref{Ystar}.
\begin{remark}\label{remBGandrates} a)
It is well known (and readily checked) that the detailed balance equations for~$\pi_*^{(N)}$, expressed in terms of the potential function $U= U^{(N)}$ (defined in~\eqref{defU}) turn into the Boltzmann-Gibbs relations
\begin{equation}\label{BG}
\frac {\pi_*^{(N)}(j)}{\pi_*^{(N)}(i)} = \frac{\lambda_i}{\lambda_j}e^{-(U(j)-U(i))}, \quad 0\le i < j \le N.
\end{equation}
\indent b) Under assumption~\eqref{smallmut} we have
\begin{equation}\label{cspeed}
    \lambda_n \sim \mu_n \sim \rho\mathfrak a/2 \quad \mbox{ provided } |n-\mathfrak a| = o(\mathfrak a)\mbox{ as } N\to \infty.
\end{equation} 
This is immediate because then the second summands in  \eqref{bestupproof} and~\eqref{bestdownproof} are asymptotically negligible compared to the first ones. The asymptotics~\eqref{cspeed} together with the time change $t \mapsto t\mathfrak c_N$ appearing in~\eqref{OUconvdef} also explains the role of $\sigma = \sigma_N=\sqrt{\rho \mathfrak a \mathfrak c}$ defined in~\eqref{defb}. 

c)
For the $\sigma$-scale around $\mathfrak a$, an inspection of~\eqref{bestupproof} gives for all $K \in \R$ and sufficiently large $N$ the identity
\begin{equation}\label{lambdaonsigma}
  \frac{\lambda_{\lfloor \mathfrak a+ \sigma K\rfloor}}{\lambda_{\lfloor \mathfrak a\rfloor }} = \Big( 1+ \frac{\rho K}{\sqrt{u_N}} \Big) \Big( 1- \frac{(1-\rho)K}{\sqrt{u_N}} \Big) \quad \mbox{as } N\to \infty.  
\end{equation}  
\end{remark}
\begin{proposition}\label{asnormpi}
{\footnotesize}Under assumptions~\eqref{smallmut} and~\eqref{expreg}
   the sequence of equilibrium distributions $\pi_*^{(N)}$ is  asymptotically normal (with mean $\mathfrak a_N$ and variance $\sigma_N^2/2)$ in the following sense:
\begin{equation}\label{loclimstrong}
    \pi_*^{(N)}(\lfloor\mathfrak a_N+K\sigma_N\rfloor) \sim \frac 1{\sqrt \pi \sigma_N}e^{-K^2} \quad \mbox{as } N\to \infty,
\end{equation}
and
\begin{equation}\label{global}
    \mbox{the image of } \pi_*^{(N)} 
\mbox{under the transformation } n \mapsto \frac {n-\mathfrak a}\sigma \mbox{converges weakly to } \mathcal N(0, 1/2).
\end{equation} 
\end{proposition}
\begin{proof}
From~\eqref{asympeN}    we know that
$$\E_{\mathfrak a}[T_0^*]:= \E\left[T_0^*\, \mid \,Y_0^{(N)}(0) = \lfloor \mathfrak a \rfloor \right]
        \sim e_N$$
Using Corollary 2.8 on p.34 of~\cite{aldous-fill-2014} we get (suppressing the floors around $\mathfrak a$ in the rest of this proof)
   \begin{equation}\label{weigtincenter}
       \pi_*^{(N)}(\mathfrak a) = \frac 1{\E_{\mathfrak a}[T_0^*]+\E_{0}[T_{\mathfrak a}^*] } \frac 1{\lambda_{\mathfrak a }+\mu_{ \mathfrak a }} \frac 1{\P_{\mathfrak a-1}(T_0 < T_{\mathfrak a})/2}\sim \frac 1{\sigma \sqrt \pi},
   \end{equation}
    where for the last asymptotics we used~\eqref{asympeN} together with Lemma~\ref{lemtimetoa}, as well as~\eqref{cspeed} and Lemma~\ref{chalem}~A). 

    For $K>0$ we obtain from~\eqref{BG}, \eqref{lambdaonsigma} and Proposition~\ref{Uquad} that
    \begin{equation}\label{deML}
   \frac{\pi_*^{(N)}(\mathfrak a + K\sigma)}{\pi_*^{(N)}(\mathfrak a)}=  \frac{\lambda_{\mathfrak a+ \sigma K}}{\lambda_{\mathfrak a}}e^{-(U(\mathfrak a+ K\sigma) - U(\mathfrak a))} \to e^{-K^2}.
\end{equation}
Combining this with~\eqref{weigtincenter} we obtain~\eqref{loclimstrong} for $K>0$.
The proof of~\eqref{loclimstrong} for $K<0$ is completely analogous. Since Proposition~\ref{Uquad} together with~\eqref{lambdaonsigma} ensures that the convergence in~\eqref{deML} is uniform on compacts as a function of $K$, the limit assertion~\eqref{global} then follows by standard arguments.   
\end{proof}
\subsection{Asymptotic normality of the quasi-equilibria: Proof of \mbox{Theorem~\ref{thmY0}.\ref{thmY0:c}}}
We first derive a rough ``large deviation'' bound for the quasi-equilibrium distribution.
\begin{lemma}\label{lemqconc}
    The quasi-equilibrium $\alpha_N$ of $Y_0^{(N)}$ obeys
    \begin{equation}\label{qconcold}
\alpha_N(\{\lfloor \mathfrak a/2\rfloor, \lfloor \mathfrak a/2\rfloor+1, \ldots, N \})  \to 1 \quad \mbox{as } N\to \infty.
    \end{equation}
\end{lemma}
\begin{proof}  Preparing for an application of Lemma~\ref{lemtimetoa} later in the proof, we define 
\begin{equation}\label{deftN}
    t_N:= \mathfrak c_N(\log \mathfrak a_N)^2. 
\end{equation}
  Because of~\eqref{QED},  the l.h.s. of~\eqref{qconcold} equals
  \begin{align}  \label{Pan} \P_{\alpha_N^{}}\Big(Y_0(t_N)\ge \frac{\mathfrak a}2\, \big | \, T_0 > t_N\Big) = \frac{\sum_{j=1}^N\alpha_N^{}(j)\P_j(Y_0(t_N)\ge \frac{\mathfrak a}2)}{\sum_{j=1}^N\alpha_N(j)\P_j(Y_0(t_N)> 0)}=: \beta_N \le 1.
  \end{align}
  The proof will be accomplished if we can show that $\beta_N\to 1$ as $N\to \infty$.
  Now assume that $\liminf \beta_N <1$, i.e. there exists an $\varepsilon > 0$ and a sequence $(N_k)$ converging to $\infty$ for which $\beta_{N_k}\le 1-\varepsilon$. This implies that for all $k$ there exists some $i_k\in [N_k]$ for which 
  $$\P_{i_k}\left(Y_0^{(N_k)}(t_{N_k})\ge \mathfrak a_{N_k}/2\right) \le (1-\varepsilon) \P_{i_k}\left(Y_0^{(N_k)}(t_{N_k})> 0\right).  $$
To lead this to a contradiction we will
prove that for any sequence $(j_N)$ 
\begin{equation}\label{singleratios}
    \frac{\P_{j_N}(Y_0(t_N)\ge \frac{\mathfrak a}2)}{\P_{j_N}(Y_0(t_N)> 0)} = \P_{j_N}\left(Y_0(t_N) \ge \frac{\mathfrak a}2 \, \big |\, T_0 > t_N\right) \to 1.
\end{equation}
To show~\eqref{singleratios} we consider the two subsequences of $(j_N)$ that are above (respectively below) $\mathfrak a_N$:

(i) For $j_N \ge \mathfrak a_N$,  the  numerator (and hence also the denominator) of the l.h.s. of~\eqref{singleratios} converges to~1. To see this, note that the event $\{T_{\mathfrak a} > t_N\}$ implies the event $\{Y_0(t_N)\ge \mathfrak a/2\}$, while the event $\{T_{\mathfrak a} \le t_N\}$ allows a re-start at time  $T_{\mathfrak a}$. We claim that then $Y_0$ remains with high probability above $\mathfrak a/2$ until time $t_N$. To see this, recall from~\eqref{largenough} that
$$\E_{\lfloor \mathfrak a\rfloor}[T_{\lfloor \mathfrak a-\sigma\rfloor}] \ge \mathfrak c/8.$$
We know from    from Lemma~\ref{chalem}\ref{maj_proba_hittingenum},   that
\begin{equation}\label{wanted} \mathbf P_{\left\lfloor \mathfrak a-\sigma\right\rfloor}(T_{\mathfrak a/2}< T_{\mathfrak a}) \leq e^{-\rho^2(\sqrt{u_N}/(2\rho)-2)}.  
\end{equation}
The estimate~\eqref{wanted} implies that
$$ \mathbf P_{\left\lfloor \mathfrak a\right\rfloor}(T_{\mathfrak a/2}<t_N) \leq  \frac{8t_N}{\mathfrak c} e^{-\rho^2(\sqrt{u_N}/(2\rho)-2)} =  8(\log \mathfrak a_N)^2 e^{-\rho^2(\sqrt{u_N}/(2\rho)-2)}=o(1) \quad \mbox{ as }N \to \infty.  $$

(ii) For $j_N < \mathfrak a_N$, the process $Y_0$ started in $j_N$ and conditioned not to hit $0$ by time $t_N$ is stochastically lower-bounded  by the process $Y_*$ (started in $j_N$) that was introduced in Definition~\ref{Ystar}. As proved in~Lemma~\ref{lemtimetoa}, we have $\E_0[T_\mathfrak a^*]\ll t_N$ as $N\to \infty$. This shows that $$\P_{j_N}\left(T_{\mathfrak a} \le \frac{t_N}2 \, \Big |\, T_0 > t_N\right) \to 1 \quad \mbox{ as } N\to \infty.$$ 
Thus, by a re-start of $Y_0$ in time $T_\mathfrak a$,  part (i)  is applicable  to ensure the convergence~\eqref{singleratios}. 
\end{proof}
Recall from Definition~\ref{Ystar} the definition of the recurrent process $Y_*^{(N)}$ and its equilibrium distribution $\pi_*^{(N)}$. In the spirit of a Doeblin coupling (cf ~\cite[Sec. 14.1.1.]{aldous-fill-2014}), we define the joint distribution of  $(Y_0^{(N)}, Y_*^{(N)}) $ as follows. Starting in the product distribution $\alpha_N \otimes \pi_*^{(N)}$, the two processes evolve independently up to their first meeting time, which we denote by $T_{\rm meet}=T_{\rm meet}^{(N)}$. As before, $T_0=T_0^{(N)}$ denotes the time at which $Y_0^{(N)}$ hits $0$ for the first time. On the event $\{T_{\rm meet}  \le T_0\}$, the process $(Y_0^{(N)}(t))_{T_{\rm meet}  \le t \le T_0}$ is taken as an identical copy of $(Y_*^{(N)}(t))_{T_{\rm meet}  \le t \le T_0}$. After time $T_0$,  the process $Y_0^{(N)}$ remains in $0$, while $Y_*^{(N)}$ continues to follow the dynamics specified in Definition~\ref{Ystar}.

As in the proof of Lemma~\ref{lemqconc} we set $t_N:= \mathfrak c_N(\log \mathfrak a_N)^2$.  

\begin{lemma}\label{meetprob}   With $t_N$ as in~\eqref{deftN}, the just  defined meeting and absorption times satisfy
 \begin{equation}\label{timesorder}
   \P(T_{\rm meet}^{(N)} < t_N < T_0^{(N)}) \to 1 \quad \mbox{ as } N\to \infty.  
 \end{equation}
\end{lemma}
\begin{proof}
     Because of~\eqref{qconcold} it suffices to show that for all sequences $(j_N)$ with  $j_N \ge \tfrac {\mathfrak a_N}2$ we have
    \begin{equation} \label{approach}
        \P(T_{\rm meet}^{(N)} < t_N < T_0^{(N)} \mid Y_0^{(N)}(0) = j_N) \to 1 \quad \mbox{ as } N\to \infty. 
    \end{equation}
    Wth $T_{\mathfrak a}= T_{\mathfrak a}^{(N)}$  denoting the time at which $Y_0^{(N)}$ hits $\mathfrak a$ for the first time, we know (under the specified starting conditions) from \eqref{comedown} and \eqref{dursweep} that
    \begin{equation}\label{TatN}
        T_\mathfrak a/t_N \to 0 \quad\mbox{in probability as } N\to \infty .
    \end{equation}
With $\mathcal H_N^*$ being defined in terms of $Y_*^{(N)}$ in the same way as $\mathcal H_N$ was obtained from $Y_0^{(N)}$ in~\eqref{OUconvdef}, we define for each $N$ two sequences of stopping times $(\rho_{N,i})_{i\in \N_0}$ and $(\rho_{N,i}^*)_{i\in \N_0}$ inductively as follows:
$$\rho_{N,0}:= T_{\mathfrak a}^{(N)}/\mathfrak c_N,\quad \rho^*_{N,0}:= \min\{t\ge \rho_{N,0}\mid \mathcal H_N^*(t) = 0\}, $$ 
$$\rho_{N,i}^{} := \min\{t>\rho_{N,i-1}^{}\mid \mathcal H_N(t) = 0 \mbox{ and } \mathcal H_N(u) \le -1 \mbox { for some } u \in (\rho_{N,i-1}, t)\},$$
$$\rho_{N,i}^{*} := \min\{t>\rho_{N,i-1}^{*}\mid \mathcal H_N^*(t) = 0 \mbox{ and } \mathcal H_N^*(u) \le -1 \mbox { for some } u \in (\rho_{N,i-1}^{*}, t)\}.$$
In addition, we define two sequences of random times $(\tau_{N,i})_{i\in \N}$ and $(\tau_{N,i}^*)_{i\in \N}$ as follows:
$$\tau_{N,i}^{} := \max\{t<\rho_{N,i}^{}\mid \mathcal H_N(t-) = 0\}, \quad \tau_{N,i}^{*}:= \max\{t<\rho_{N,i}^{*}\mid \mathcal H_N^*(t-) = 0\}. $$
Finally, we define
$$J^N:= \min\{j\ge 1 \mid \exists\, i \mbox { such that } \tau_{N,j} \le \tau_{N,i}^* \le \rho_{N,j} \le \rho_{N,i}^\ast  \}$$
We claim that
\begin{equation}\label{Tmeetbound}
T_{\rm meet}^{(N)} \le \rho_{N,J^N_{}}^{}.
\end{equation}
Indeed, if $i$ is a natural number that obeys $\tau_{N,J^N} \le \tau_{N,i}^* \le \rho_{N,J^N} \le \rho_{N,i}^\ast$, then the excursion of $\mathcal H_N$ that happens between times $\tau_{N,J^N}$ and  $\rho_{N,J^N}$ necessarily meets the excursion of $\mathcal H_N^*$ that happens between times $\tau_{N,i}^*$ and  $\rho_{N,i}^*$.

Because of~\eqref{Tmeetbound} (and~\eqref{asympeN}), the assertion of the Lemma follows if we can show 
\begin{equation} \label{aim}
    \P(\rho_{N,J^N} \le t_N) \to 1 \quad\mbox{ as } N\to \infty.
\end{equation}
To this purpose we observe that,  as a consequence of Theorem~\ref{thmY0}.\ref{thmY0:b}, 
\begin{equation}\label{HNconv}
\mbox{the sequence of processes } (\mathcal H_N(\rho_{N,0}^{} + t))_{t \ge 0} \quad \mbox{converges, as } N\to \infty, \mbox{ in distribution to } \mathcal H.
\end{equation}
The same assertion is true for the processes $(\mathcal H_N^*(\rho_{N,0}^{*} + t))_{t \ge 0}$.

In analogy to $\rho_{N,i}^{}$ and $\rho_{N,i}^*$ we define a sequence of $\mathcal H$-measurable stopping times $(\rho_i)_{i\in \N_0}$ inductively as follows:
\begin{equation}\label{defrhoi}
 \rho_0:= 0, \qquad \rho_i^{} := \min\{t>\rho_{i-1}^{}\mid \mathcal H(t) = 0 \mbox{ and } \mathcal H(u) \le -1 \mbox { for some } u \in (\rho_{i-1}, t)\}.   \end{equation} 
Also, in analogy to $\tau_{N,i}^{}$ and $\tau_{N,i}^*$, we define the sequence of random times $(\tau_i)_{i\in \N}$ as follows: 
$$\tau_{i}^{} := \max\{t<\rho_{i}^{}\mid \mathcal H(t) = 0\}. $$
As a consequence of~\eqref{HNconv} we obtain that the sequence of random sequences $$\mathscr P_N:=\left(\tau_{1,N}, \, \rho_{1,N},\, \tau_{2,N},\,\rho_{2,N},\,\tau_{3,N}, \, \rho_{3,N} \ldots \right)    $$
converges, after a backshift by $\rho_{0,N}$, as $N\to \infty$ in distribution to the random sequence
$$\mathscr P:= \left(\tau_1, \, \rho_1,\, \tau_2, \, \rho_2, \, \tau_3, \, \rho_3, \ldots\right).$$
Likewise, we obtain that the sequence of random sequences 
$$\mathscr P_N^*:=\left(\tau_{1,N}^*, \, \rho_{1,N}^*,\, \tau_{2,N}^*,\,\rho_{2,N}^*,\,\tau_{3,N}^*, \, \rho_{3,N}^* \ldots \right)    $$ converges, now after a backshift by $\rho_{0,N}^*$, as $N\to \infty$ in distribution to $\mathscr P$. Because of the indepencence of $\mathcal H_N$ and $\mathcal H_N^*$ up to their meeting time, we can consider a random sequence $$\mathscr P^*=\left(\tau_1^*, \, \rho_1^*,\, \tau_2^*, \, \rho_2^*, \, \tau_3^*, \, \rho_3^*, \ldots\right)$$ which arises from an i.i.d. copy of $\mathscr P$ after a  random shift
 that is independent of $\mathscr P$. The random variable 
 $$J:= \min\{j\ge 1 \mid \exists\, i \mbox { such that } \tau_{j} \le \tau_{i}^* \le \rho_{j} \le \rho_{i}^\ast  \}$$
 thus figures as an asymptotic stochastic upper bound for the sequence $(J^N)$ as $N\to \infty$. (Note that  on certain events $J^N$ may remain strictly smaller than $J$,  e.g. when  the first meeting of $\mathcal H_N$ and $\mathcal H_N^*$ happens in excursions {\em above} $\mathfrak a$.)
 It follows from basic properties of $\mathcal H$ that the distribution of $J$ 
 has geometric tails. We thus obtain~\eqref{aim}, because for each fixed $N$   the increments $\rho_{N,i}-\rho_{N,i-1}$, $i=1,2,\ldots$ are (as long as $\rho_{N,i} < T_0^{(N)}$) i.i.d. copies of  $\rho_{N,1}-\rho_{N,0}$, which converges as $N\to \infty$ in distribution to $\rho_1$ defined in~\eqref{defrhoi}.
\end{proof}
\begin{remark} 
The choice of $(t_N)$ in~\eqref{deftN} admits modifications. For  specific  choices of starting values $j_N^{}$ (like e.g.  $j_N= \frac {\lfloor \mathfrak a_N\rfloor}2$) one might ask for sequences $(t_N)$ that satisfy~\eqref{approach} and are ``asymptotically as small as possible''. This points into the direction of questions studied e.g.~in~\cite{barbour2025convergence}, for which the sequences $Y_0^{(N)}$ and $Y_*^{(N)}$ might provide an interesting case.
\end{remark}

\begin{proof}[Completion of the proof of Theorem~\ref{thmY0}.\ref{thmY0:c}] In the above-defined coupling we have as a consequence of~\eqref{timesorder} that
$$\P(Y_0^{(N)}(t_N) \neq Y_*^{(N)}(t_N)) \to 0 \quad \mbox{as } N\to \infty.$$
This implies that the variation distance between $\pi_*^{(N)}$ and the distribution of $Y_0^{(N)}(t_N)$ tends to $0$, provided $Y_0^{(N)}$ is started in $\alpha_N$.
Because of~\eqref{QED}, and since $\P(T_0 > t_N) \to 1$, also the variation distance between $\alpha_N$ and $\P_{\alpha_N}(Y_0^{(N)}(t_N) \in (\cdot))$ tends to $0$ as $N\to 0$. Consequently, we have $d_{\rm TV}(\alpha_N, \pi_*^{(N)}) \to 0$ as $N\to \infty$. Combined with Proposition~\ref{asnormpi} b)  this completes the proof of Theorem~\ref{thmY0}.\ref{thmY0:c}.
\end{proof}

\subsection{Asymptotic exponentiality of the extinction time: Proof of Theorem~\ref{thmY0}.\ref{thmY0:d}}\label{metast} 
First we we will show 
\begin{proposition}\label{qeexpconv} Let $T_{(\alpha_N,0)}$ be the extinction time of  $Y^{(N)}_0$ when started in its quasi-equilibrium distribution $\alpha_N$ (cf. ~eq.~\eqref{QED}).
Then the sequence $T_{(\alpha_N,0)}/e_N$  converges as $N\to \infty$  in distribution to a standard exponential random variable.
\end{proposition}
\begin{proof}  It is well known (see e.g.~\cite[Proposition 2]{MeleardVillemonais2012}) that  $T_{(\alpha_N, 0)}$ has an exponential distribution. All what remains to show is thus that  
\begin{equation}\label{QEDexp}
    \E[T_{(\alpha_N, 0)}]\sim e_N \mbox{ as } N\to \infty. 
\end{equation} 
Part c) of Theorem~\ref{thmY0} implies that 
\begin{equation*}
\alpha_N(\{1,2, \ldots \lfloor \mathfrak a/2\rfloor\})  \to 0 \quad \mbox{as } N\to \infty.
    \end{equation*}
On the other hand we know from~\eqref{asympeN} that $\E_{j_N}[T_0]\sim e_N$ for all sequences $j_N$ with $\mathfrak a/2 < j_N \le N$. Since for any fixed $N$ the mapping $n\to \E_n[T_0]$ is increasing in $n$, this implies~\eqref{QEDexp}. 
\end{proof}

 For $j_N\in \{1,\ldots, N\}$  we denote by $T_{(j_N,0)}$ the extinction time of $Y^{(N)}$ when started in $j_N$.
Let $j_N \gg \mathfrak c_N$ and $T_{(\alpha_N,0)}$  be as in Proposition~\ref{qeexpconv}.  The random times $T_{(\alpha_N,0)}$ and $T_{(j_N,0)}$ are stochastically dominated by $T_{(N,0)}$, i.e. there exist couplings 
    \begin{equation}\label{couplingTS}T_{(N,0)}= S_{(N,\alpha_N)} + T_{(\alpha_N,0)}, \quad T_{(N,0)}= S_{(N,j_N)} + T_{(j_N,0)}
    \end{equation}
    with $S_{(N,\alpha_N)}$ and $S_{(N,j_N)}$ nonnegative.\footnote{For a distinguished coupling in which  $S_{N,\alpha_N}$ is a ``time to quasi-equilibrium'' that is independent of $T_{(\alpha_N,N)}$, see Proposition 5 and the remark at the end of Section 3 of \cite{diaconis2009times}.} 
    From~\eqref{asympeN} and Lemma~\ref{lemqconc}   we know that 
    $$\E[T_{(N,0)}]\sim \E[T_{(j_N,0)}] \sim e_N \sim \E[T_{(\alpha_N,0)}].$$
    Thus, because of~\eqref{couplingTS},
    $$0\le \E\left[\frac {S_{(N,\alpha_N)}}{e_N}\right] \to 0 \quad \mbox {as } N\to \infty$$ 
    which by Markov's inequality implies that the sequence $({S_{(N,\alpha_N)}}/{e_N})$ converges to $0$ in probability. 
    A similar argument, again based on~\eqref{couplingTS}, shows that also the sequence $({S_{(N,j_N)}}/{e_N})$ converges to $0$ in probability as  $N\to \infty$.
    We know from Proposition~\ref{qeexpconv} that   the sequence $({T_{(\alpha_N,0 )}}/{e_N})$ converges in distribution to a  standard exponential random variable. Together with~\eqref{couplingTS}, with the above argument and Slutski's theorem, this shows that also the sequence $({T_{(j_N,0)}}/{e_N})$ is asymptotically standard exponential.
    This completes the proof of part d) of Theorem~\ref{thmY0}.
   \subsection{On the way to extinction: Proof of Theorem~\ref{thmY0}.\ref{thmY0:e}}\label{lastexc}

\begin{lemma}\label{improbreturn}For all $\eps < 1/3$, and for a suitably chosen $C>0$,
  with high probability the process~$Y_0$, after having left $\lfloor3\eps \mathfrak a\rfloor$ forever 

  (i) will not return to $\lfloor 2\eps \mathfrak a\rfloor$ after having visited $\lfloor\eps \mathfrak a\rfloor$, and

  (ii)  makes at least $C\eps\frac Nm$ steps   between states~$2\eps \mathfrak a$ and $0$.
\end{lemma}
\begin{proof}
    (i) Let us denote by~$W$ the discrete-time chain of a process whose law is that of the fittest class size conditioned to reach $0$ before returning to $3\eps\mathfrak a$. Let $1 \leq k \leq 2\eps \mathfrak a$ and consider the law of the next increment conditionally on the process staying below $3\varepsilon \mathfrak a$:
\begin{equation}\label{stepup}
   \P_k\Big( W(1)=k+1 | T_0 < T_{3\eps\mathfrak a} \Big)= \frac{\lambda_k}{\lambda_k+\mu_k} \frac{\P_{k+1}(T_0 < T_{3\eps\mathfrak a} )}{\P_k(T_0 < T_{3\eps\mathfrak a} )}. 
\end{equation}  
 Using the harmonicity of the function $R$ defined in~\eqref{RKdef} 
we have  
\begin{equation}\label{probratio}
    \frac{\P_{k+1}(T_0 < T_{3\eps\mathfrak a} )}{\P_k(T_0 < T_{3\eps\mathfrak a} )} = \frac{r_{k+1}+r_{k+2}+... + r_{3\varepsilon \mathfrak a-1}}{r_{k}+r_{k+1}+... + r_{3\varepsilon\mathfrak a-1}}.
\end{equation}
Noticing  that, for $l \leq \mathfrak a$, the oddratios $r_l$ are decreasing with $l$,
we obtain that

\begin{align}  \inf_{k \leq i \leq 3\varepsilon \mathfrak a-1}\frac{r_{i+1}}{r_i} - \frac{1}{3\varepsilon\mathfrak a-k-1} \label{sumratio}
&\leq 
\frac{r_{k+1}+r_{k+2}+... + r_{3\varepsilon \mathfrak a-1}}{r_{k}+r_{k+1}+... + r_{3\varepsilon\mathfrak a-1}}\le 
\sup_{k \leq i \leq 3\varepsilon \mathfrak a -1}\frac{r_{i+1}}{r_i} \end{align}
and we can sandwich the bounds in~\eqref{sumratio} by observing that for $1\le i \le 3\eps \mathfrak a$
\begin{align*}1- \frac{2m}{\rho}(1-\rho) \leq \frac{r_{i+1}}{r_i}= 1+ \frac{2m}{\rho}\Big( \frac{\rho}{1- (i+1)/N} -1\Big)
&\leq 1-  2m\Big( \frac{1}{\rho}-2\frac{i+1}{N}-1 \Big) \\
& \leq 1- \frac{2m}{\rho}(1-\rho)+13\eps (1-\rho)m,  \end{align*}
where we used that  $1/(1-x)\leq 1 + 2x$ for $x\leq 1/2$. Besides, we have, using the same inequality, for $k \leq 2\eps \mathfrak a$
$$\frac{1}{2}\Big(1+\frac{m}{\rho} ( 1- \rho) \Big)-{3}\eps{\rho}(1-\rho)m \leq \frac{\lambda_k}{\lambda_k+\mu_k} = \frac{1}{2}\Big(1+\frac{m}{\rho} \Big( 1- \frac{\rho}{1-k/N} \Big) \Big)\leq \frac{1}{2}\Big(1+\frac{m}{\rho} ( 1- \rho) \Big). $$
Combining the previous two chains of inequalities with~\eqref{stepup},~\eqref{probratio} and~\eqref{sumratio} we get 
\begin{multline*}\frac{1}{2} \left(  1+\frac{m}{\rho} ( 1- \rho) -6\eps\rho(1-\rho)m\right) \left(  1- \frac{2m}{\rho}(1-\rho)- \frac{1}{\eps \mathfrak a-1} \right)   \\
\leq   \P_k\left(   W(1)=k+1 | T_0 < T_{3\varepsilon \mathfrak a}  \right)   \leq \\ \frac{1}{2} \left(  1+\frac{m}{\rho} ( 1- \rho)  \right)  \left(   1- \frac{2m}{\rho}(1-\rho)+13\eps (1-\rho)m \right) .\end{multline*}
As for large $N$,
$$ \left( \frac{m(1-\rho)}{\rho}\right) ^2 + \frac{1}{\eps \mathfrak a-1} = o(m(1-\rho)), $$
we deduce that for $N$ large enough, 
\begin{eqnarray*}
    \frac{1}{2} \left(  1-\frac{m}{\rho} ( 1- \rho)-7\eps(1-\rho)m \right)  
&\leq &  \P_k \Big(  W(1)=k+1 | T_0 < T_{3\varepsilon \mathfrak a}  \Big)   \\&\leq& \frac{1}{2}\left(  1-\frac{m}{\rho} ( 1- \rho) +14\eps (1-\rho)m\right)  
\end{eqnarray*}
In particular, focusing on the lower bound we deduce that
\begin{align*} \P_{\eps\mathfrak a}\left(   T_{2\eps\mathfrak a}  < \infty  | T_0 < T_{3\varepsilon \mathfrak a} \right) &\leq 
\left(   \frac{1- \frac{m}{\rho}(1-\rho)(1-{7}\eps \rho)}{1+ \frac{m}{\rho}(1-\rho)(1-{7}\eps \rho)} \right) ^{\eps \mathfrak a}\\
&= 
 \left(  1-2\frac{\frac{m}{\rho}(1-\rho)(1-{7}\eps \rho)}{1+ \frac{m}{\rho}(1-\rho)(1-{7}\eps \rho)}  \right)  ^{\eps \mathfrak a}\\
& \leq \left(   1-\frac{m}{\rho}(1-\rho)(1-{7}\eps \rho) \right) ^{\eps \mathfrak a}\\
& \leq \exp \left( -\frac{\eps mN (1-\rho)^2(1-{7}\eps \rho)}{\rho} \right)   \end{align*}
which, again because of assumption~\eqref{expreg}, converges to $0$ as $N\to \infty$. This proves part (i) of the Lemma. 

To show part (ii), note that, as a consequence of the previous computations, there exists a simple random walk $W^{(-)}$ making up jumps with probability
\begin{equation}\label{expec}
  \frac{1}{2}- \frac{\mu^{(-)}}{2}:=\frac{1}{2} \left(  1- \frac{m}{\rho}(1-\rho)(1-{7}\eps \rho)  \right) 
\end{equation}  
such that with high probability, for any $n \in \N$, 
$$ W(n) \geq W^{(-)}(n). $$
Now, if we denote by $T^X_i$ the hitting time of $i$ by the random walk $X$, this coupling entails for any positive finite $C$
$$ \P_{\eps(1-\rho)N} \left(  T^W_0 > C\varepsilon \frac N{m} \right)  \geq \P_0 \left(  T^{W^{(-)}}_{-\eps (1-\rho)N} > C\varepsilon \frac N{m}  \right)  .$$
We have the following equality of  events: 
$$\left\{T^{W^{(-)}}_{-\varepsilon (1-\rho)N} <  C\varepsilon \frac N{m}\right\} = \left\{\inf_{0\le i \le C\varepsilon \frac N{m}} W^{(-)}({i}) < -\varepsilon (1-\rho)N\right\}.$$ 
\normalcolor
By assumption, 
$$M(n):= W^{(-)}(n)+n\mu^{(-)}$$
is 
a simple symmetric random walk. We thus obtain for $g \in \N$ 
\begin{equation}\mathbf P\left(\inf_{0\le i \le n}W^{(-)}(i)  {\leq} -g- n\mu^{(-)}\right) = 
\mathbf P\left(\inf_{0\le i \le n}M
(i)\le -g \right) \le \exp\left(-\frac{g^2}{2n}\right), \label{eq:boundWminuses}
\end{equation}
where the last estimate follows from the Azuma-Hoeffding inequality.  
With  the numbers of steps in \eqref{eq:boundWminuses} chosen as  
 $n:= \frac{C\varepsilon N}{m}$  we   obtain   from~\eqref{expec} that

$$n\mu^{(-)} = \frac{C(1-\rho)\varepsilon N}\rho (1-{7}\eps \rho),$$
Since the starting point was $\varepsilon (1-\rho) N$, the choice 
\begin{equation}\label{first} g+ n\mu^{(-)} = \varepsilon (1-\rho)N
\end{equation}
\normalcolor
turns into
\begin{equation}\label{secondnew}  g =\varepsilon N(1-\rho)\left(1-\frac C\rho  (1-{7}\eps \rho)\right).
\end{equation}
Because of our assumption that $\rho$ is bounded away from $0$ we may choose 
$$ C = \frac{1}{2}\liminf_{N \to \infty}\rho_N=: \frac{{\rho_*}}{2}. $$
Consequently,
$$\frac {g^2}{2n} =\varepsilon \left(1-\frac C\rho (1-5\eps \rho)\right)^2 \frac{ {\mathfrak u_N}}{2C} $$
converges to $\infty$ as $N\to \infty$ under our standing assumption~\eqref{expreg}  and thus the right hand side of \eqref{eq:boundWminuses} converges to zero. 
\end{proof}
We are now ready to prove Theorem~\ref{thmY0}.\ref{thmY0:d}.
We know from Lemma~\ref{improbreturn} that with high probability, after its last hitting of the state $3 \eps \mathfrak a$ the process $Y_0$ once it has reached $\eps\mathfrak a$ 
\begin{itemize}
\item Does not reach the state $2\eps \mathfrak a$ anymore
\item Makes a number of jumps which is at least $\eps N{\rho_*}/(2m)$ and thus a number of downward jumps which is at least $\eps N{\rho_*}/(4m)$
\normalcolor
\end{itemize}

Noticing that, when $Y_0$ is in state $k$, an individual's ``death'' is due a mutation with probability 
$$ \frac{2m}{1-k/N+2m} \geq \frac{2m}{1-\eps}, $$
we thus obtain that on $Y_0$'s way to extinction and while it is of size smaller than $2\eps\mathfrak a$, the number of mutations affecting the ``type 0'' population  is at least 
$$ 
\frac{\eps N {\rho_*}}{4m}  \cdot \frac{2m}{1-\eps}  \gg \frac{1}{m(1-\rho)}$$
where the last estimate follows because  $Nm(1-\rho) \ge Nm(1-\rho)^2 \gg 1$, cf. our standing condition~\eqref{expreg}. $\Box$

\section{Proof of Theorem~\ref{th:clicktimepois}}
\label{sec4}
In this section we complete the proof of Theorem~\ref{th:clicktimepois}. The core idea is to proceed inductively and to show that with high probability the fittest class at the time of its disappearance has left a number $\gg \mathfrak c$ of mutants, which helps establishing the ``next fittest class''. 
\subsection{Lower-bounding the size of the new fittest class at a clicktime
}
The following key proposition uses the hierarchical autonomy of the tournament ratchet stated in Remark~\ref{autonomy}; its proof builds on Theorem~\ref{thmY0}.\ref{thmY0:e}.
\begin{proposition}\label{prop:sizenewbestclass}
Let  $(Y_0^N,Y_1^N)$ be a bivariate birth-and-death process whose jump rates are the same as those of $(\mathfrak N^{(N)}_\kappa, \mathfrak N^{(N)}_{\kappa+1})$ specified in Definition~\ref{tfpdyn} on the event $\{K_N^\star(t) = \kappa\}$. Assume there exists a sequence $(j_N)$ with $j_N\gg \mathfrak c_N$ such that
$$\P(Y_0^{(N)}(0) \ge j_N) \to 1 \quad\mbox{ as } N\to \infty.$$
Then there exists a sequence $(g_N)$ with $g_N \gg \mathfrak c_N$ such that 
\begin{equation}
    \P\left(Y^{(N)}_1(T_0) \ge g_N\right) \to 1 \quad\mbox { as } N\to \infty,
\end{equation}
where $T_0= T_0^{(N)}$ is the hitting time of $0$ of the process $Y_0^{(N)}$.
\end{proposition}

\begin{proof} With regard to Remark~\ref{autonomy} we consider a bivariate birth-and-death process $(Y_0,Y_1)$ whose jump rates are the same as those of $(\mathfrak N_\kappa, \mathfrak N_{\kappa+1})$ specified in Definition~\ref{tfpdyn} on the event $\{K^\star(t) = \kappa\}$. In particular,
as long as $(Y_0, Y_1)$ is in state $(n_0,n_1)$  the upward jump rate of $Y_1$ is
\begin{equation*}
 \mathfrak{b}(n_0,n_1):= m n_0 + n_1 \cdot \left[ \frac12 \left(1- \frac{n_1}{N}\right) + \frac{m}{\rho } \left( 1- \frac{n_0}{N} - \frac{n_1}{N}\right) \right]
\end{equation*}
and the downward jump rate is given by
\begin{equation*}
   \mathfrak{d}(n_0,n_1):=  n_1 \cdot \left[\frac12 \left(1- \frac{n_1}{N}\right)  +m  + \frac{m}{\rho }  \cdot \frac{n_0}{N}\right].
\end{equation*}
In accordance with Theorem~\ref{thmY0}.\ref{thmY0:d} we denote by $\mathfrak P=[L,  T_0\,)$ the period between the time at which~$Y_0$ visits $2\eps\mathfrak a$ for the last time and the time of the extinction of $Y_0$. In the next arguments we will make use of the fact that this period cannot last too long; more specifically (see~\eqref{durlasttravel}),  $\mathbf E[T_0-L] =O(\mathfrak c \log \mathfrak a)$.  We now claim:

(C) If during $\mathfrak P$ the process $Y_1$ ever reaches $ \mathfrak a/4$, then with high probability it will not drop down below $\mathfrak a/8 \gg  \mathfrak c$  by the end of $\mathfrak P$ (which corresponds to the time of the next click).

To prove claim (C) we \normalcolor
introduce a birth-and-death process $\widetilde{Y}_1$  whose birth and death rates in state $n_1 \in \N$ are 
\begin{equation*}
 \mathfrak{b}(n_1):=  n_1\beta(n_1)= n_1\cdot \left[ \frac12 \left(1- \frac{n_1}{N}\right) + \frac{m}{\rho } \left( 1- \frac{2\eps \mathfrak a}{N} - \frac{n_1}{N}\right) \right],
\end{equation*}
\begin{equation*}
   \mathfrak{d}(n_1):= n_1\delta(n_1)= n_1 \cdot \left[\frac12 \left(1- \frac{n_1}{N}\right)  +m  + \frac{m}{\rho }  \cdot \frac{2\eps \mathfrak a}{N}\right].
\end{equation*}
 Define $T:= \min \left( \mathfrak P \cap \{t\mid Y_1(t)\le \mathfrak \lfloor \mathfrak a\rfloor /4\}\right)$. Then on the event $\{T<\infty\}$ the processes $Y_1$ and $\widetilde Y_1$ can be coupled such that
$$Y_1(T) = \widetilde Y_1(T) \,\mbox{ and }\, Y_1(t) \ge \widetilde Y_1(t)  \, \mbox{ a.s. for } t\in \mathfrak P \cap [T, \infty). $$
To ease notation, let us shift time and assume (again omitting the floor brackets for convenience) that 
$$ \widetilde{Y}_1(0) = Y_1(0) = \frac{\mathfrak a}{4}.  $$

Let us denote by $W$ a discrete random walk making $+1$ jumps with probability $p$ and $-1$ jumps with probability $q=1-p$. Then if we take $a,b \in \N$ and denote by $T_k^W$ the walk's $W$ first hitting of $K$, we know that 
$$ \P (T^W_{-a}<T^W_b) = \frac{(p/q)^b-1}{(p/q)^{a+b}-1} .$$
Let $\mathfrak a/6 \leq k \leq \mathfrak a/2 $. Then 
\begin{align*}
    \frac{\beta_k}{\beta_k + \delta_k}\geq \frac{\beta_{\mathfrak a/2}}{\beta_{\mathfrak a/2}+\delta_{\mathfrak a/2}} = \frac{1}{2}+ \frac{m}{2\rho}\Big( \frac{(1+\rho)/2-4\eps (1-\rho)-\rho}{1+m/\rho+m} \Big) \geq \frac{1}{2}+ \frac{m}{8\rho}(1-\rho)
\end{align*}
for $\eps := 1/24$ and  $N$ large enough. In particular, 
$$ \frac{\beta_k}{\delta_k} \geq 1+ \frac{m}{2\rho}(1-\rho), $$
$$   \frac{\beta_k}{\beta_k + \delta_k}-  \frac{\delta_k}{\beta_k + \delta_k} \geq \frac{m}{4\rho}(1-\rho) $$
and if we denote by $\tilde{T}_k$ the hitting time of $k$ by $\tilde Y_1$ we obtain for $N$ large enough
$$ \P_0\big(\tilde T_{-\mathfrak a/12} < \tilde T_{\mathfrak a/4} \big) \leq \Big(1+ \frac{m}{2\rho}(1-\rho)\Big)^{-\mathfrak a/12} \leq e^{-(\mathfrak a/12)m(1-\rho)/2\rho}= e^{-u_N/24\rho}. $$
Let us now consider the sequence $(\mathscr P_i, i \in \N)$ of successive parts of paths of the process $\tilde Y_1$ with initial state $\mathfrak a/4$ and which come back to $\mathfrak a/4$ after having visited $\mathfrak a/2$ and not $\mathfrak a/6$. According to the previous result, there are more than a geometric of parameter $ e^{-u_N/24\rho}$ of such paths before reaching $\mathfrak a/6$. Then each $\mathscr P_i$ is made of at least $\mathfrak a/12$ steps of independent exponential random variables of intensity at most $2\mathfrak a$. Let us denote by $|\mathscr P_i|$ the duration of the path $\mathscr P_i$. We will prove that 
$$ \P \big( |\mathscr P_i| <1/36 \big)\leq e^{- \mathscr Q \mathfrak a} $$
where
$$ \mathscr Q = 1/36+\log 2 - \log 3 <0. $$
For this, we denote by $\mathcal E_i$ a sequence of independent exponential random variables with parameter~$2 \mathfrak a$. We get, using Markov inequality
\begin{align*}
    \P \Big( |\mathscr P_i| <\frac{1}{36} \Big)\leq \P \Big( \sum_{i=1}^{\mathfrak a/12} \mathcal{E}_i <\frac{1}{36} \Big) & =  \P \Big( e^{- \mathfrak a\sum_{i=1}^{\mathfrak a/12} \mathcal{E}_i} >e^{-\frac{\mathfrak a}{36}} \Big)\\
    & \leq e^{\frac{\mathfrak a}{36}}\E \Big[e^{- \mathfrak a\sum_{i=1}^{\mathfrak a/12} \mathcal{E}_i}\Big]\\
    & \leq e^{\frac{\mathfrak a}{36}}\Big( \frac{2}{3} \Big)^{\mathfrak a/12}= e^{- \mathscr Q \mathfrak a} .
\end{align*}
We thus have the following properties:
\begin{enumerate}
    \item The period $\mathfrak P$ has an expected duration $O(\mathfrak c \log \mathfrak a)$ (see~\eqref{durlasttravel}).
    \item If during $\mathfrak P$ the process $Y_1$ ever reaches $ \mathfrak a/4$, then with a probability larger than $1-e^{-u_N/30\rho}$ it will reach $\mathfrak a/2$ and come back to $\mathfrak a/4$ before reaching $\mathfrak a/6$ and this excursion will take at least a time $1/36$.
\end{enumerate}
Since
$$ \mathfrak c \log \mathfrak{a } e^{-u_N/24\rho} \leq u_N e^{-u_N/24\rho} \to 0 \quad \mbox{ as } N \to \infty,$$
this concludes the proof of claim (C).

We now consider the events
$$G_1:= \{Y_1( L) > \mathfrak a/4\}, \, G_2:= \{Y_1( L) \le \mathfrak a/4\}, \,F:= \{Y_1(t) = \mathfrak a/4 \mbox{ for some } t \in \mathfrak P\}.$$
Because of Claim (C) we have $Y_1(T_0) \gg \mathfrak c$ with high probability on the event $F$, and obviously we have $Y_1(T_0) >\mathfrak a/4 \gg \mathfrak c$ on the event $G_1\cap F^c$. 

It remains to consider the event  $G_2\cap F^c$, on which $Y_1$ does not exceed $\mathfrak a/4$ during the period $\mathfrak P$.

We know from Theorem~\ref{thmY0}.\ref{thmY0:e} that the number of mutants that ``immigrate'' into the second fittest class during period $\mathfrak P$ is with high probability $\gg \mathfrak c$.
The difference between the upward and the downward jump rates of the second fittest class is
\begin{equation}\label{eq:driftprop3.1}
\mathfrak{b}(n_0,n_1)-\mathfrak{d}(n_0,n_1)=   m n_0 + \frac{mn_1}{\rho N} \left[ N  (1-\rho)-2n_0-n_1 \right].
  \end{equation}
Thus for $\eps \le \frac 18$, as long as the size of the fittest class  is smaller than $2\eps\mathfrak a$  and the size of the second fittest class is $\le \mathfrak a/4$, the supercriticality of the second fittest class is $\ge \frac 12 m(1-\rho) = \frac 1{2\mathfrak c}$. Hence on the event $G_2\cap F^c$ we can lower-bound $Y_1$ by a  branching process with supercriticality $\ge \frac 1{2\mathfrak c}$  and a number of immigrants during period $\mathfrak P$ that is $\gg \mathfrak c$. This process will reach a size $\gg \mathfrak c$ with high probability, which allows to conclude the proof of the Proposition. \end{proof} 

\subsection{From one click to the next} Let $\mathcal T^{(i)}_N$, $i\ge 1$,  be as in~\eqref{def_Ti}, with $\mathcal T^{(0)}_N:= 0$.

(a) 
 We show by  induction: For all $i\ge 0$ there exists a sequence $(j_N)$ with $j_N \gg \mathfrak c_N$ such that
\begin{equation}\label{newmany}
   \P\left(\mathfrak N_i^{(N)}(\mathcal T^{(i)}_N) \ge j_N \right)\to 1 \quad\mbox{ as } N \to \infty.   
 \end{equation}
 Indeed, for $i=0$ this follows from assumption~\eqref{incondnew}, while the induction step is a direct corollary of Proposition~\ref{prop:sizenewbestclass} combined with Remark~\ref{autonomy}.

 (b) By definition the times $\mathcal T_N^{(i)}$ are the jump times of the process $K_N^\ast$. Since \, $\mathfrak N_{i-1}^{(N)}(\mathcal T^{(i)}_N)=0$ \, by construction,  the events  $\{K_N^\star(\mathcal T_N^{(i)})=i\}$ and $\{\mathfrak N_0^{(N)\star}(\mathcal T^{(i)}_N) = \mathfrak N_i^{(N)}(\mathcal T^{(i)}_N)\ge j_N\}$ are implied by the event appearing in~\eqref{newmany}. Again employing Remark~\ref{autonomy}, we can thus apply Theorem~\ref{thmY0}.\ref{thmY0:d}  combined with~\eqref{newmany}  to conclude by induction that    $ \left( \mathcal T_N^{(i)}-T_N^{(i-1)}\right)\big /{e_N} $, ${i\in \mathbb N}$, converges in distribution as $N\to \infty$ 
 to a sequence of idependent standard exponential random variables.\\

\paragraph{\bf Funding} This work was partially funded by the Chair "Modélisation Mathématique
et Biodiversité" of VEOLIA-Ecole Polytechnique-MNHN-F.X. and by CNRS via an International Emerging Action (IEA) project. 
\bibliography{bibratchet} 
\bibliographystyle{habbrv}
\end{document}